\documentclass[11pt,reqno]{amsart}
\usepackage{a4wide}
\usepackage{amssymb}
\usepackage{amsmath}
\usepackage{amsthm,hyperref}
\usepackage{mathrsfs}

\numberwithin{equation}{section}

\theoremstyle{definition}
\newtheorem*{definition*}{Definition}
\newtheorem{definition}{Definition}

\theoremstyle{remark}
\newtheorem{remark}{Remark}

\newtheorem*{note*}{Note}

\newtheorem*{example*}{Example}

\newtheorem*{question*}{Question}
\newtheorem*{blank*}{}

\theoremstyle{plain}
\newtheorem{theorem}{Theorem}
\newtheorem*{theorem*}{Theorem}

\newtheorem*{corollary*}{Corollary}
\newtheorem{lemma}{Lemma}
\newtheorem*{lemma*}{Lemma}
\newtheorem{proposition}{Proposition}
\newtheorem*{proposition*}{Proposition}

\newtheorem*{conjecture*}{Conjecture}

\newtheorem*{assump*}{Assumption}
\newtheorem*{claim*}{Claim}

\numberwithin{theorem}{section}
\numberwithin{lemma}{section}
\numberwithin{proposition}{section}

\newcommand{\ba}{\mathbb A}

\newcommand{\bc}{\mathbb C}

\newcommand{\bn}{\mathbb N}

\newcommand{\bq}{\mathbb Q}
\newcommand{\br}{\mathbb R}

\newcommand{\bz}{\mathbb Z}

\newcommand{\ma}{\mathcal A}

\newcommand{\mf}{\mathcal F}

\newcommand{\ml}{\mathcal L}

\newcommand{\mo}{\mathcal O}

\newcommand{\mr}{\mathcal R}
\newcommand{\ms}{\mathcal S}

\newcommand{\fg}{\mathfrak g}

\newcommand{\fl}{\mathfrak l}

\newcommand{\fp}{\mathfrak p}

\newcommand{\fs}{\mathfrak s}

\newcommand{\msr}{\mathscr R}

\newcommand{\Rar}{\Rightarrow}

\newcommand{\rar}{\rightarrow}

\newcommand{\lrar}{\longrightarrow}

\newcommand{\dlim}{\varinjlim}

\newcommand{\geqs}{\geqslant}
\newcommand{\leqs}{\leqslant}
\newcommand{\wtilde}{\widetilde}

\newcount\colveccount
\newcommand*\colvec[1]{
        \global\colveccount#1
        \begin{pmatrix}
        \colvecnext
}
\def\colvecnext#1{
        #1
        \global\advance\colveccount-1
        \ifnum\colveccount>0
                \\
                \expandafter\colvecnext
        \else
                \end{pmatrix}
        \fi
}

\newcommand{\smalltwomatrix[5]}{
  \left( \begin{smallmatrix}
            {#2} & {#3}\\
            {#4} & {#5}
         \end{smallmatrix}\right)}

 \DeclareMathOperator\Hom{Hom}
 
\DeclareMathOperator\Res{Res}   \DeclareMathOperator\Ind{Ind}
     \DeclareMathOperator\ad{ad}

\DeclareMathOperator\re{Re}   

\DeclareMathOperator\Tr{Tr} 
 \DeclareMathOperator\sgn{sgn}

\newcommand{\leftup}[2]{{^{#1}}\mspace{-2.5mu}#2}

\newcommand{\isoto}{\stackrel{\sim}{\lrar}}

\newcommand{\dd}{\mathrm{d}}

\input{xy}
\xyoption{all}

\newcommand{\cl}{\mathrm{cl}}
\newcommand{\SO}{\mathrm{SO}}

\newcommand{\OO}{\mathrm{O}}

\newcommand{\Sp}{\mathrm{Sp}}

\newcommand{\GL}{\mathrm{GL}}
\newcommand{\PGL}{\mathrm{PGL}}
\newcommand{\SL}{\mathrm{SL}}
\newcommand{\SU}{\mathrm{SU}}

\newcommand{\tildesr}{\wtilde{\msr}}

\begin{document}

\title{The Whittaker period formula on metaplectic $\SL_2$}
\author[ Yannan Qiu]{Yannan Qiu}
\address{Department of Mathematics and Statistics, University of Maine, Neville Hall, Orono, ME 04469.}
\email{yannan.qiu@gmail.com} 
\begin{abstract}
The Whittaker period formula on $\wtilde{\SL}_2(\ba_F)$ was previously established only when the base field $F$ is totally real. We present a new simple proof that works for all base number fields. Our local argument is uniform at every local place of $F$, based on the isometry property of quadratic Fourier transform and the estimates of matrix coefficients and Whittaker functions imposed by the unitarity of the local representations.

\end{abstract}
\maketitle

\setcounter{tocdepth}{1}
\tableofcontents

The Whittaker period formula on metaplectic $\SL_2$ generalizes Waldspurger's classical formula \cite{wald81} relating the Fourier coefficients of half integral weight modular forms to the central $L$-values of integral weight modular forms.

Let $F$ be a number field, $\ba:=\ba_F$ be the ring of $F$-adeles, and $\wtilde{\SL}_2(\ba)$ be the two-fold metaplectic cover of $\SL_2(\ba)$. For a genuine irreducible cuspidal automorphic representation $\sigma=\otimes \sigma_v$ of $\wtilde{\SL_2}(\ba)$ and a character $\psi$ of $\ba/F$, the global Whittaker period functional $\ell_{\sigma,\psi}$ is
\begin{equation}\label{gwf1}
\ell_{\sigma,\psi}(\varphi)=\int_{\ba/F}\varphi\smalltwomatrix{1}{x}{}{1}\psi(-x)\dd x,\quad \varphi\in \sigma.
\end{equation}

On the other hand, a local Whittaker period functional can be formally obtained by integrating matrix coefficients,
\begin{equation}\label{lwf1}
\ml_{\sigma_v,\psi_v}(\varphi_{1,v},\varphi_{2,v})=\int_{F_v}\big(\sigma_v\smalltwomatrix{1}{x_v}{}{1}\varphi_{1,v},\varphi_{2,v}\big)_{\sigma_v}\psi_v(-x_v)\dd x_v,\quad \varphi_{1,v}, \varphi_{2,v}\in \sigma_v.
\end{equation}
Here $(\, , \,)_{\sigma_v}$ denotes the local Hermitian pairing on $\sigma_v$ and, for compatibility, their product is required to be equal to the standard global Hermitian pairing $(\, ,\, )_\sigma$ on $\sigma$ given by
\[
(\varphi_1,\varphi_2)_\sigma=\int_{\SL_2(F)\backslash \SL_2(\ba)}\varphi_1(g)\overline{\varphi_2(g)}\dd g,\quad \varphi_1, \varphi_2\in \sigma.
\]
When the local representation $\sigma_v$ is square-integrable, the integral in \eqref{lwf1} is convergent for all $\varphi_{1,v},\varphi_{2,v}\in \sigma_v$. But if $\sigma_v$ is not square-integrable, a regularization is needed to make sense of the right hand side of (\ref{lwf1}). We choose to do so in terms of distribution theory: the bounded smooth function $(\sigma_v\smalltwomatrix{1}{x_v}{}{1}\varphi_{1,v},\varphi_{2,v})_{\sigma_v}$ defines a tempered distribution on the space of Bruhat-Schwartz functions $\ms(F_v)$ and the Fourier transform of this distribution can be shown to be represented by a smooth function $W_{\varphi_{1,v},\varphi_{2,v},\psi_v}(a_v)$ on $F_v^\times$. We define
\begin{equation}\label{intro_reg}
\int_{F_v}(\sigma_v\smalltwomatrix{1}{x_v}{}{1}\varphi_{1,v},\varphi_{2,v})_{\sigma_v}\psi_v(-x_v)dx_v:=W_{\varphi_{1,v},\varphi_{2,v},\psi_v}(-1).
\end{equation}
We show in Lemma ~\ref{m2wsl2} that the local period functional $\ml_{\sigma_v,\psi_v}$ defined in this way is nonzero if and only if $\Hom_{F_v}(\sigma_v,\psi_v)$ is nonzero, where $F_v$ is identified with the unipotent subgroup of $\wtilde{\SL}_2(F_v)$.

It is known that $\Hom_{F_v}(\sigma_v,\psi_v)$ is at most one dimensional. So when the global functional $\ell_{\sigma,\psi}\otimes \overline{\ell_{\sigma,\psi}}$ is nonzero, it must be the product of the local functionals $\ml_{\sigma_v,\psi_v}$ up to a constant. It is the task of the Whittaker period formula to connect them with an explicit constant. When doing so, the measures on $\ba$ and $\SL_2(\ba)$ need to be specified---we choose the Tamagawa measure, with respect to which both $\ba/F$ and $\SL_2(F)\backslash \SL_2(\ba)$ have volume $1$. For general metaplectic groups $\wtilde{\Sp}_{2n}$, a similar problem exists and a precise Whittaker period formula was conjectured by Lapid and Mao \cite{lapid_mao2015} for genuine non-exceptional cuspidal automorphic representations of $\wtilde{\Sp}_{2n}(\ba)$.

\textbf{The main purpose} of this article is to give a new simple proof of the Whittaker period formula on $\wtilde{\SL}_2$ for an arbitrary base number field. Previously, the formula is only established when the base field is totally real, due to technical issues at archimedean places. To state the formula, we distinguish Weil representations of $\wtilde{\SL}_2(\ba)$, which are small representations consisting of elementary theta series, from genuine cuspidal automorphic representations that are orthogonal to them. 
\begin{theorem*}
Let $\sigma$ be a genuine irreducible cuspidal automorphic representation of $\wtilde{\SL}_2(\ba)$ that is orthogonal to elementary theta series. Suppose that the Whittaker functional $\ell_{\sigma,\psi}$ is nonzero and write $\pi=\Theta_{\wtilde{\SL}_2\times \PGL_2}(\sigma,\psi)$ for the global theta lift of $\sigma$ to $\PGL_2(\ba)$ with respect to $\psi$. Then for decomposable vectors $\varphi_i=\otimes \varphi_{i,v}\in \sigma$, there is
\[
\ell_{\sigma,\psi}(\varphi_1)\overline{\ell_{\sigma,\psi}(\varphi_2)}=\frac{1}{2}\cdot \frac{L(\frac{1}{2},\pi)\zeta_F(2)}{L(1,\pi, \ad)}\prod_v \frac{L(1,\pi_v,\ad)}{L(\frac{1}{2},\pi_v)\zeta_{F_v}(2)}\int_{F_v}\big(\sigma_v\smalltwomatrix{1}{x_v}{}{1}\varphi_{1,v},\varphi_{2,v}\big)_{\sigma_v}\psi_v(-x_v)dx_v.
\]
\end{theorem*}
\noindent Note that when the global period functional $\ell_{\sigma,\psi}$ is nonvanishing, the theta lift $\Theta_{\wtilde{\SL}_2\times \PGL_2}(\sigma,\psi)$ is automatically nonzero and irreducible. Note also that the normalized local factors on the right hand side is equal to $1$ for almost all $v$, whence the infinite product is actually a finite product.\\

In general, the proof of a period formula needs two types of ingredients: (I) an algebraic tool providing the context in which the global problem is reduced to a local problem such as local matching or a local identity, (II) an analytic tool that actually proves the local identity. Let us make a few remarks:\\

(a) When $F$ is totally real, the above period formula can be deduced from Theorem 4.1 in Baruch and Mao's paper \cite{baruch_mao07}, although the regularization of local integrals is not discussed there. In their work, the algebraic context is provided by the relative trace formula while the local analytic ingredient is a local identity between the relative Bessel distribution on $\GL(2)$ and the standard Bessel distribution on $\wtilde{\SL}_2$ over nonarchimedean and real fields \cite{baruch_mao03, baruch_mao05}. The same approach for the local identity over the complex field is technically more difficult; as Baruch told the author, the regularity problem of the two mentioned distributions at the complex place has not been worked out yet and he currently has an answer only for the standard Bessel distribution on $\SL_2(\bc)$. This accounts for the restriction of $F$ being totally real in \cite{baruch_mao07}.\\

(b) Assuming that $F$ is totally real and that the local components of $\sigma$ at the real places of $F$ are square-integrable, the above formula can also be deduced in the algebraic context provided by the descent method of Ginzburg-Rallis-Soudry. Under this assumption, the local integrals at archimedean places do not need regularization; with the nonarchimedan local integrals regularized via the notatoin of stable integrals, Lapid-Mao have written a sequence of nice papers \cite{lapid_mao2013, lapid_mao2014, lapid_mao2014b, lapid_mao2016}, culminating in the Whittaker period formula for non-exceptional cuspidal automorphic representations of $\wtilde{\Sp}_{2n}$. The key local analytic ingredient in their argument is a local identity of iterated integrals. However, at the real and complex places, the convergence issue of integrals are more serious and it is not clear how to prove the local identity when $\sigma_v$ is non-square-integrable, even on $\wtilde{\SL}_2$. This accounts for the assumption made on $F$ and the archimedean components of $\sigma$ in \cite{lapid_mao2016}.\\

Our proof of the Whittaker period formula on $\wtilde{\SL}_2$ uses the theta correspondence between $\wtilde{\SL}_2$ and $\PGL_2$ to set up the algebraic context. The validity of the local integral transform we use relies only on the estimates of matrix coefficients and Whittaker functions derived from the unitarity of local representations. Because such estimates are uniform over nonarchimedean and archimedean local fields, our local argument is accordingly uniform, yielding the formula unconditionally in a very pleasant manner.

Here is the outline of the argument: With theta correspondence, we transfer the global Whittaker period integral of $\pi$ to a product of local integrals over $\wtilde{\SL}_2(F_v)$, and, transfer (via the inner product formula concerning the lifting from $\wtilde{\SL}_2$ to $\PGL_2$) the local Whittaker period integral of matrix coefficients of $\pi_v$ to a local integral over $\SL_2(F_v)$. We then compare the two types of local integrals in the context of $\wtilde{\SL}_2(F_v)$---they are shown to be proportional to each other. The local proportionality constant is determined by the local inner product on $\sigma_v$ and the choice of the local Whittaker functional on $\sigma_v$, and it is no wonder that the product of these constants is essentially the constant $c_\sigma$ showing up in the Whittaker period formula of $\sigma$. This shows that the number $c_\sigma$ fits into the Whittaker period formula of $\pi$. On the other hand, one can prove the Whittaker period formula of $\pi$ directly and get the constant $c_\pi=\frac{1}{2}$. Therefore , $c_\sigma=c_\pi$ takes value $\frac{1}{2}$ and we obtain the Whittaker period formula of $\sigma$.

In the key step of comparing the two types of local integrals, one technically needs an analytic ingredient to transform one integral into the other. This ingredient is Lemma \ref{identity-mbintegral}, about integrating the matrix coefficient of $\sigma_v$ against the quadratic Fourier transform of a Bruhat-Schwartz function along the unipotent subgroup of $\SL_2(F_v)$. It is an application of the isometry property of quadratic Fourier transform (c.f. Lemma ~\ref{isometry-identity}).We note that for such an isometry to be applicable to $\sigma_v$, the matrix coefficients and Whittaker functions of $\sigma_v$ need to asymptotically behave well. It turns out that the asymptotic estimates imposed by the unitarity of $\sigma_v$ are just sufficient. We have observed the following local mechanism:
\begin{itemize}
\item[(i)] The regularized local period integral (\ref{intro_reg}) can be used as a canonical choice of the local Whittaker functional. With it, the asymptotic estimate of Whittaker functions can be deduced from the asymptotic estimate of matrix coefficients (See the proof of Lemma ~\ref{sl2wbound}). 
\item[(ii)] The inverse integral transform of the local period integral, which is a quadratic Fourier transform in this article, behaves well (See Lemma ~\ref{lemmaw2m}, ~\ref{identity-mbintegral}) when the matrix coefficients and the Whittaker functions satisfy the asymptotic estimates imposed by unitarity.
\end{itemize}
For general reductive or covering groups, one often makes extra effort to derive the asymptotic estimates of local period functionals, which is related to Plancherel type theorems. Our view is that matrix coefficients are fundamental objects in local harmonic analysis and their asymptotic estimate or expansion can naturally lead to an according estimate of general local period functionals. \textbf{The secondary purpose} of this article is to demonstrate such a principle on $\wtilde{\SL}_2$, which already shows some aspects of the general situation. 

Parallel to the period transfer investigated in this article, one may transfer the global Whittaker period of $\sigma$ and the local Whittaker integral of matrix coefficient of $\sigma_v$ to $\SO_2$-integrals of $\SO_3$-representations, via the theta correspondence between $\wtilde{\SL}_2$ and $\SO_3$; for example, \cite{mao2012} considers the the case of anisotropic $\SO_3$. In general, the Fourier-Jacobi period of a genuine automorphic representation of $\wtilde{\Sp}_{2n}$ is related to the Gross-Prasad type of period of automorphic representations of $\SO_m$.

In the end of this introduction, we mention that the proof in this article is nearly self-contained, except (standard facts about $\GL_2$, $\wtilde{\SL}_2$ representations and) the inner product formula concerning the global theta lifting from $\wtilde{\SL}_2$ to $\PGL_2$ (c.f. Proposition \ref{prop_ip}), the proof of which is available in \cite{qiu14}, p.6714-p.6719.\\

\noindent \textbf{Acknowledgement} This work was done in 2013 during the author's visit to National University of Singapore. The author would like to thank NUS for providing great working environment and to thank Prof. Gan Wee Teck for helpful conversations. The author also thanks the referee for a careful reading and several comments that help improve the exposition of the paper.

\setcounter{section}{-1}
\section{Notations}\label{notation}

Let $F$ be a number field and $\ba:=\ba_F$ be the ring of $F$-adeles. When $v$ is a  non-archimedean place of $F$, let $\mo_{F_v}$ denote the ring of $v$-adic integers in $F_v$ and $\varpi_v$ be a uniformizer of $\mo_{F_v}$. For a non-trivial character $\psi_v$ of $F_v$, define its conductor to be the largest $\varpi^n_v\mo_{F_v}$ on which $\psi_v$ is trivial. For a quasi-character $\chi_v$  of $F_v^\times$, define its conductor $\mathrm{Cond}(\chi_v)$ to be $\mo_{F_v}$ if $\chi_v$ is unramified, and to be $\varpi^n\mo_{F_v}$ if $\chi_v$ is ramified and $ n$ is the smallest positive integer such that $\chi_v$ is trivial on $1+\varpi^n_v\mo_{F_v}$.

Let $\psi:\ba/F\rar S^1$ be a non-trivial character. At each place $v$ of $F$, let $da_v$ be the self-dual Haar measure of $F_v$ with respect to $\psi_v$, then $da=\prod_v da_v$ is the Tamagawa measure on $\ba$ and does not depend on $\psi$. Write $d^\times a_v=\frac{da_v}{|a_v|}$ and $d^\times a=\prod_v d^\times a_v$. The Tamagawa measure on $\ba^\times$ is
\[
d^\ast a=\frac{1}{\Res_{s=1}\zeta_F(s)}\prod_v \zeta_{F_v}(1)d^\times a_v.
\]
We usually decompose $d^\ast a=\prod_v d^\ast a_v$, with $d^\ast a_v$ such that $\int_{\mo_{F_v}^\times}1 d^\ast a_v=1$ for almost all $v$.

For $\delta\in F^\times$, write $\psi_\delta=\psi(\delta\cdot)$ for the $\delta$-twist of $\psi$ and $\chi_\delta$ for the quadratic character $<\delta,\cdot>$ on $\ba^\times$, where $<,>$ denotes the Hilbert symbol.\\

Let $k$ be a local field of characteristic zero and $m\in \bn$. Define the norm $|\cdot|_k$ by $d(ax)=|a|_kdx$, where $dx$ is any Haar measure. The subscript $k$ will be dropped if the context is clear. Note specifically that the complex norm is $|a|_\bc=a\bar{a}$ ($a\in \bc$), which is different from the usual norm.

For a quasi-character $\chi:k^\times \rar \bc^\times$, define its exponent $e(\chi)$ to be the real number such that $\chi\overline{\chi}=|\cdot|^{2e(\chi)}$. Write the unitary part of $\chi$ as $\chi_0=\chi|\cdot|^{-e(\chi)}$.

The space  $\ms(k^m)$ of Bruhat-Schwartz functions on $k^m$ consists of smooth functions all of whose derivatives are rapidly decreasing if $k$ is archimedean, or locally constant compactly supported functions if $k$ is non-archimedean. Its topology is defined below:
\begin{itemize}
\item[(i)] When $k=\br$,  consider semi-norms $\Vert\cdot\Vert_{p,q}$, where $\alpha, \beta\in (\bn\cup\{0\})^m$ are multi-indices and 
\[
\Vert f\Vert_{\alpha,\beta}=\sup_{x\in \br^m} \Big|x^{\beta}\cdot \frac{\partial^{\alpha}f}{\partial x^{\alpha}}\Big|.
\]
The topology on $\ms(\br^m)$ is such that $f_n\rar f$ if and only if $\Vert f_n-f\Vert_{\alpha,\beta}\rar 0$ for all $\alpha,\beta$.
\item[(ii)] When $k=\bc$, identify $\ms(\bc^m)$ with $\ms(\br^{2m})$ and give the according topology.
\item[(iii)] When $k$ is nonarchimedean, let $\ms_n(k^m)$ denote the space of functions $f$ on $k^m$ satisfying (a) $f(x_1)=f(x_2)$ if $|x_1-x_2|\leqs \frac{1}{2^n}$ and (b) $f(x)=0$ if $|x|>2^n$, then $\ms_n(k^m)$ is finite dimensional for each $n$ and $\ms(k^m)=\dlim_n \ms_n(k^m)$ is an inductive limit. Give each $\ms_n(k^m)$ the standard topology and $\ms(k^m)$ the accordingly inductive limit topology.
\end{itemize}
For a vector space $V$ over a number field $F$, define $\ms(V_\ba)$ as the restricted tensor product of $\ms(V_{F_v})$ with respect to the characteristic functions $1_{V(\mo_{F_v})}$.

We fix a few terms. A smooth function on an open subset $U$ of $k^m$ is one that can be differentiated for infinitely many times if $k$ is archimedean or one that is locally constant if $k$ is nonarchimedean. $C^\infty(U)$ denotes the space of smooth functions on $U$ and $C^\infty_c(U)$ denotes the subspace of compactly supported smooth functions on $U$. Note that  $C^\infty_c(k^m)=\ms(k^m)$ if $k$ is nonarchimedean.

\section{Fourier Transform and Quadratic Fourier Transform}\label{qftsec}

Let $k$ be a local field of characteristic zero and choose a uniformizer $\varpi$ if $k$ is nonarchimedean. Let $\psi:k\rar \mathrm{S}^1$ be a nontrivial character and $dx$ be the according self-dual Haar measure of $k$. The main results in this section are Proposition ~\ref{isometry-identity} and ~\ref{isometry-identity-extra}, about the isometry of quadratic Fourier transform.

\subsection{Fourier Transform}

Fourier transform $\mf$ is originally defined on $L^1(k)$,
\begin{equation}
f(x)\overset{\mf}{\lrar} \hat{f}(x):=\int_k f(y)\psi(xy)dy.
\end{equation}
It satisfies the following standard properties:
\begin{itemize}
\item[(i)]  $\hat{f}$ is uniformly continuos.
\item[(ii)] $\underset{x\rar \infty}\lim \widehat{f}(x)=0$.
\item[(iii)] $\mf$ is a homeomorphism from $\ms(k)$ onto $\ms(k)$.
\item[(iv)] $\hat{\hat{f}}(x)=f(-x)$ if $\hat{f}\in L^1(k)$ and $f$ is continuous.
\item[(v)] If $f_i\in L^1(k)\cap L^2(k)$ ($i=1,2$), then $\hat{f}_i\in L^2(k)$ and $\int_k \hat{f}_1(x)\hat{f}_2(-x)dx=\int_k f_1(x)f_2(x)dx$.
\end{itemize}

\indent Let $\ms^\prime(k)$ denote the space of continuous linear functionals $T:\ms(k)\rar \bc$. Note that if $k$ is nonarchimedean, then a linear functional $T$ on $\ms(k)$ is automatically continuous with respect to the inductive limit topology on $\ms(k)$ and $\ms^\prime(k)$ is simply the space of linear functionals on $\ms(k)$. We call elements in $\ms^\prime(k)$ the tempered distributions on $k$.

Fourier transform can be extended to $\ms^\prime(k)$ by setting $\mf T=T\mf$. When $k$ is archimedean, a constant coefficient differential operator $D$ on $\ms(k)$ can also be extended to $\ms^\prime(k)$ by setting $DT=TD$.

\begin{definition}\label{defdistribution}
A functional $T\in \ms^\prime(k)$ is said to be represented over an open subset $U\subset k$ by a function $f\in L^1_{loc}(U)$ if $T(\phi)=\int_{U}f(x)\phi(x)dx$ for all $\phi\in C^\infty_c(U)$. Obviously, if $f$ is required to be continuous, then it is unique once it exists.

A functional $T\in \ms^\prime(k)$ is said to be supported on a closed subset $E$ of $k$ if $T(\phi)=0$ for all $\phi\in C^\infty_c(E^c)$, where $E^c$ refers to the complement of $E$ in $k$.
\end{definition}

Denote by $\delta_a$ the Dirac distribution supported at $a\in k$: $\delta_a(\phi)=\phi(a), \phi\in \ms(k)$. It is tempered.

\begin{lemma}\label{pointsupport}
Suppose $T\in \ms^\prime(k)$ is supported at a single point $a\in k$, then it is a multiple of $\delta_a$ if $k$ is nonarchimedean or a finite linear combination of derivatives of $\delta_a$ if $k$ is archimedean.
\end{lemma}
\begin{proof}
When $k$ is archimedean, the lemma follows from \cite[Theorem 2.3.4]{hormander03}. When $k$ is nonarchimedean, choose one function $f\in \ms(k)$ with value $1$ at $a$. Then for any $\phi\in \ms(k)$, $\phi-\phi(a)f$ is zero near $a$, whence $T(\phi)=T(\phi(a)f)+T(\phi-\phi(a)f)=\phi(a)T(f)$ and therefore $T=T(f)\delta_a$.
\end{proof}

\subsection{Quadratic Fourier Transform}

Quadratic Fourier transform $\mf_2$ is defined on $L^1(k)$ by
\begin{equation}\label{quadraticf}
\mf_2 f(t)=\int_k f(x)\psi(tx^2)dx.
\end{equation}
When $f(x)$ is an odd function, $\mf_2f=0$. When $f(x)$ is an even function, $\mf_2$ can be related to $\mf$ in the following way: associates to $f$ a function $\wtilde{f}(x)$ supported on ${k^\times}^2\cup\{0\}$,
\begin{equation*}
\wtilde{f}(x)=
\begin{cases}
f(\sqrt{x}), &x\in {k}^2,\\
0, &x\not\in {k}^2,
\end{cases}
\end{equation*}
then
\begin{equation}\label{ff}
\mf_2f=2|2|_k^{-1}\mf(\wtilde{f}|\cdot|^{-1/2}).
\end{equation}

Due to the quadratic oscillation $tx^2$ in the defining equation (\ref{quadraticf}), the analytic property of quadratic Fourier transform differs from that of Fourier transform. Formally applying the $L^2$-isometry property of Fourier transform, one may formally get the according isometry property of quadratic Fourier transform, namely,
\begin{equation}\label{quadratici}
\int_k \mf_2f_1(t)\mf_2f_2(-t)dt=2|2|_k^{-1}\int_k f_1(x)f_2(x)|x|^{-1}dx.
\end{equation}
However, the analytic condition for this isometry property needs to be carefully examined.

For our intended use of (\ref{quadratici}), the function $f_2$ is in $\ms(k)$; accordingly, $\mf_2 f_2(t)$ is roughly of size $f_2(0)|t|^{-\frac{1}{2}}$ near infinity and is not square-integrable if $f_2(0)\neq 0$.  In order to have (\ref{quadratici}), we impose the following natural conditions on $f_1$:
\begin{itemize}
\item $f_1, f_1|\cdot|^{-1}\in L^1(k)$,
\item $\mf_2 f_1\in L^2(k)$ and $\mf_2 f_1|\cdot|^{-1/2}\in L^1(k)$.
\end{itemize}
The task of this section is to argue that under these conditions on $f_1$, (\ref{quadratici}) holds for all $f_2\in \ms(k)$.

\subsubsection{The asymptotic approximation of $\widehat{\phi \chi}$}\label{subsec_asym}

Suppose $\phi\in \ms(k)$ and $\chi\in \Hom(k^\times, \bc^\times)$ with $e(\chi)>-1$. We derive in this subsection the leading term and the according error bound in the asymptotic expansion of $\widehat{\phi\chi}(x)$. A full asymptotic expansion can be worked out but is not pursued here.

Our tool is Tate's local zeta integral \cite[Section 2]{tate1967}. The zeta integral
\[
Z(\phi,\chi,s)=\int_{k^\times} \phi(x)\chi(x)|x|^sd^\times x,\quad s\in \bc,
\]
is absolutely convergent when $\re(s)>-e(\chi)$, has meromorphic continuation to $\bc$, and satisfies
\begin{equation}\label{localfe0}
\frac{Z(\hat{\phi},\chi^{-1},1-s)}{L(1-s,\chi^{-1})}=\epsilon(s,\chi,\psi)\frac{Z(\phi,\chi,s)}{L(s,\chi)}.
\end{equation}
Here $\epsilon(s,\chi,\psi)$ is of the form $ab^s$ and is never zero. We may rewrite the functional equation as
\begin{equation}\label{localfe}
Z(\hat{\phi},\chi^{-1},\psi^{-1})=\gamma(s,\chi,\psi)Z(\phi,\chi,s),
\end{equation}
with $\gamma(s,\chi,\psi)=\epsilon(s,\chi,\psi)\cdot \frac{L(1-s,\chi^{-1})}{L(s,\chi)}$

$\ $\\
\indent We view $\widehat{\phi\chi}(x)$ as a local zeta integral,
\begin{equation}\label{zeta1}
\widehat{\phi\chi}(x)=\int_k \phi(y)\chi(y)\psi(xy)dy=Z(\phi\psi_x,\chi,s)|_{s=1}.
\end{equation}

Suppose $x\neq 0$. In the region $\re(s)+e(\chi)<1$, we apply (\ref{localfe}) to get the following expression
\begin{align}\label{zeta2}
\nonumber Z(\phi\psi_x,\chi,s)&=\gamma(1-s,\chi^{-1},\psi^{-1})Z(\widehat{\phi\psi_x},\chi^{-1},1-s)\\
\nonumber &=\gamma(1-s,\chi^{-1},\psi^{-1})\int_k \hat{\phi}(x+y)\chi^{-1}(y)|y|^{-s}dy\\
&=\frac{\gamma(1-s,\chi^{-1},\psi^{-1})}{\chi(x)|x|^s}\cdot \int_k \hat{\phi}(y)\chi^{-1}(\frac{y}{x}-1)\big|\frac{y}{x}-1\big|^{-s}dy.
\end{align}

When $k$ is nonarchimedean, (\ref{zeta2}) suffices for deriving an asymptotic expansion when $|x|$ is large.
\begin{lemma}\label{zetapadic}
Suppose $k$ is nonarchimedean. If $\mathrm{Supp}(\widehat{\phi})\subseteq x\cdot \big(\mathrm{Cond}(\chi)\cap \varpi \mo_k\big)$, then \[
Z(\phi\psi_x,\chi,s)=\chi(-1)\phi(0)\gamma(1-s,\chi^{-1},\psi^{-1})\chi^{-1}(x)|x|^{-s}.
\]
\end{lemma}
\begin{proof}
When $\mathrm{Supp}(\widehat{\phi})\subseteq x\cdot \big(\mathrm{Cond}(\chi)\cap \varpi \mo_k\big)$, there is $\chi(\frac{y}{x}-1)=\chi(-1)$ and $|\frac{y}{x}-1|=1$, whence
\[
\int_k \hat{\phi}(t)\chi^{-1}(\frac{y}{x}-1)\big|\frac{y}{x}-1\big|^{-s}dt=\int_k \hat{\phi}(t)dt=\phi(0).
\]
By (\ref{zeta2}), in the region $\re(s)<1-e(\chi)$, there is
\[
Z(\phi\psi_x,\chi,s)=\phi(0)\gamma(1-s,\chi^{-1},\psi^{-1})\chi^{-1}(x)|x|^{-s}.
\]
Since both sides are meromorphic in $s$, this equality must hold on the whole complex $s$-plane. 
\end{proof}

When $k$ is archimedean, the integral in (\ref{zeta2}) may diverge at $y=x$ if $\phi$ is nonvanishing at $y=x$ and $\re(s)+e(\chi)\geqs 1$. If this is the situation one faces with, it is good to localize the integral at $y=x$ and apply the functional equation to the localized integral again. We proceed as follows.

Suppose $k$ is archimedean and choose a non-negative smooth cut-off function $\alpha(t)$ on $k$ satisfying
\begin{equation*}
\alpha(t)=\begin{cases}
1, &|t|\leqs \frac{1}{16},\\
0, &|t|>\frac{1}{8}.
\end{cases}
\end{equation*}
By the triangle inequality, the supports of $\alpha(\frac{y}{x})$ and $\alpha\big(\frac{y}{x}-1\big)$ are disjoint. We then write
\begin{equation}
\int_k \hat{\phi}(y)\chi^{-1}(\frac{y}{x}-1)\big|\frac{y}{x}-1\big|^{-s}dy=Y_1(x ; s)+Y_2(x ; s)+Y_3(x ; s),
\end{equation}
with
\begin{align}
Y_1(x ; s)=&\int_k \chi^{-1}(\frac{y}{x}-1)\big|\frac{y}{x}-1\big|^{-s}\cdot \hat{\phi}(y)\alpha\big(\frac{y}{x}\big)dy,\label{y1}\\
Y_2(x ; s)=&\int_k \chi^{-1}(\frac{y}{x}-1)\big|\frac{y}{x}-1\big|^{-s}\cdot \hat{\phi}(y)\alpha\big(\frac{y}{x}-1\big)dy,\label{y2}\\
Y_3(x ; s)=&\int_k \chi^{-1}\big(\frac{y}{x}-1\big)\big|\frac{y}{x}-1\big|^{-s}\cdot \hat{\phi}(y)\Big(1-\alpha\big(\frac{y}{x}\big)-\alpha\big(\frac{y}{x}-1\big)\Big)dy. \label{y3}
\end{align}

The integrals in (\ref{y1}) and (\ref{y2}) are always convergent and $Y_1(x ; s), Y_3(x ; s)$ are holomorphic in $s$. $Y_2(s)$ is the localized integral that may have a trouble when $\re(s)+e(\chi)\geqs 1$. We are mainly interested in the region $\re(s)+e(\chi)>0$. So we make a change of variable $t=\frac{y}{x}-1$ in (\ref{y2}) and apply the functional equation (\ref{localfe}) for a second time---this leads to
\begin{equation}\label{y22}
Y_2(x; s)=|x|\gamma(s,\chi,\psi)\int_k \chi(y)|y|^s\Phi_x(y)d^\times y,\quad \re(s)+e(\chi)>0,
\end{equation}
where $\Phi_x(y)=\int_k \alpha(t)\hat{\phi}\big(x(t+1)\big)\psi(-ty)dt$. To summarize, we have obtained
\begin{equation}\label{zeta3}
Z(\phi\psi_x,\chi,s)=\frac{\gamma(1-s,\chi^{-1},\psi^{-1})}{\chi(x)|x|^s}\sum_{i=1}^3 Y_i(x;s).
\end{equation}
We show below that $Y_2(s)$ and $Y_3(s)$ are essentially negligible and $Y_1(s)$ is the main term.

\begin{lemma}\label{lemmay3}
Suppose $|x|\geqs 1$ and $\re(s)+e(\chi)\in [e_1,e_2]$ with $e_1>0$.\\
\emph{(i)} For any $A>0$, there exists a constant $C_A$ depending on $A, \phi, e_2$ such that $|Y_3(x ; s)|\leqs C_A|x|^{-A}$.\\
\emph{(ii)} For any $A>0$, there exists a constant $C_A$ depending on $A, \phi, e_1$ such that
\[
|\gamma(1-s,\chi^{-1},\psi^{-1})Y_2(x ; s)|\leqs C_A|x|^{-A}.
\]
\end{lemma}
\begin{proof}
(i) For $Y_3(x;s)$, we use (\ref{y3}). In order that the integrand therein is nonzero, it is necessary that $\big|\frac{y}{x}-1\big|\geqs \frac{1}{8}$, whence $|y-x|\geqs \frac{1}{8}|x|$ and therefore
\[
|\chi^{-1}\big(\frac{y}{x}-1\big)\big|\frac{y}{x}-1\big|^{-s}|= |x|^{\re(s)+e(\chi)}|y-x|^{-\re(s)-e(\chi)}\leqs 8^{\re(s)+e(\chi)}\leqs 8^{e_2}.
\]
It follows that
\[
|Y_3(x ; s)|\leqs 8^{e_2}\int_{|y|\geqs \frac{|x|}{8}} \big|\hat{\phi}(y)\big|dy\leqs C_A|x|^{-A},
\]
where $C_A$ depends on $A, \phi, e_2$ and we have used the fact $\widehat{\phi}$ is rapidly decreasing.

(ii) For $Y_2(x;s)$, we use (\ref{y22}), which leads to
\[
\gamma(1-s,\chi^{-1},\psi^{-1})Y_2(x ; s)=|x|\int_k \chi(y)|y|^s \Phi_x(-y)d^\times y.
\]
Recall $\Phi_x(-y)=\int_k \alpha(t)\hat{\phi}\big(x(t+1)\big)\psi(-ty)dt$. Because $\widehat{\phi}$ and its derivatives are rapidly decreasing, there are constant $c_{\phi,A}$ and $c_{\phi,n,A}$ such that
\begin{align*}
|\Phi_x(-y)|&\leqs c_{\phi,A}|x|^{-A},\quad\quad\quad\,\, \text{for all $y\in k$,}\\
|\Phi_x(-y)|&\leqs c_{\phi,n,A}|x|^{-A}y^{-n},\quad |y|\geqs 1, n\in \bn.
\end{align*}
Choose $n> e_2$, then
\begin{align*}
\Big|\int_k \Phi_x(-y)\chi(y)|y|^sd^\times y\Big|&\leqs \int_{|y|\leqs 1}c_{\phi,A}|x|^{-A}|y|^{e_1}d^\times y+\int_{|y|\geqs 1}c_{\phi,n,A}|x|^{-A}|y|^{e_2-n}d^\times y \leqs C_A|x|^{-A}.
\end{align*}
where $C_A$ depends on $\phi, A, e_1$. This leads to the assertion in (ii).
\end{proof}

The uniform estimate of $Y_1(x;s)$ needs to be treated more carefully. If $k=\br$, write $\chi(x)=\sgn(x)^m|x|^r$ with $m\in \{0,1\}$. If $k=\bc$, write $\chi(x)=|x|_\bc^r x^m\bar{x}^n$, with $m, n\in \bz_{\geqs 0}$ and $mn=0$.
\begin{lemma}\label{lemmay1}
Suppose $|r+s|\leqs e$ if $k=\br$ or $|r+s|+m+n\leqs e$ if $k=\bc$. There exists a constant $C$ depending on $\phi$ and $e$ such that $Y_1(s)=\chi(-1)\phi(0)+R(x;s)$, with 
\begin{equation}
|R(x;s)|\leqs 
\begin{cases}
C|x|_\br^{-1}, & k=\br,\\
C|x|_\bc^{-\frac{1}{2}}, &k=\bc.
\end{cases}
\end{equation}
\end{lemma}
\begin{proof}
We use (\ref{y1}). In order the integrand therein to be nonzero, it is necessary that $\big|\frac{y}{x}\big|\leqs \frac{1}{8}$. So one may use the binomial expansion of $\chi^{-1}\big(\frac{y}{x}-1\big)\big|1-\frac{y}{x}\big|^{-s}$ to derive the asymptotic expansion of $Y_1(x;s)$. We take the leading term only and write
\[
\chi^{-1}\big(\frac{y}{x}-1\big)\big|1-\frac{y}{x}\big|^{-s}=\chi(-1)+\mr\big(\frac{y}{x}\big),
\]
where $|\mr(\frac{y}{x})|\leqs C_0\big|\frac{y}{x}|$ if $k$ is real or $|\mr(\frac{y}{x})|\leqs C_0\big|\frac{y}{x}|^{-1/2}$, with $C_0$ dependent on $e$. It follows that 
\begin{align*}
Y_1(x ; s)=&\chi(-1)\int_k  \hat{\phi}(y)\alpha\big(\frac{y}{x}\big) dy + \int_k  \hat{\phi}(y)\alpha\big(\frac{y}{x}\big)\mr\big(\frac{y}{x}\big)dy\\
=&\chi(-1)\int_k \hat{\phi}(y)dy+\chi(-1)\int_k \widehat{\phi}(y)\big(1-\alpha\big(\frac{y}{x}\big)\big)dy+\int_k  \hat{\phi}(y)\alpha\big(\frac{y}{x}\big)\mr\big(\frac{y}{x}\big)dy.
\end{align*}
The first term is $\chi(-1)\phi(0)$. The second term is bounded by a multiple of $|x|^{-A}$ for any $A>0$ because $\widehat{\phi}$ is rapidly decaying. The third term is bounded by a multiple of $|x|^{-1}_\br$ or $|x|^{-1/2}_\bc$ due to the bound of $\mr\big(\frac{y}{x}\big)$ and the rapid decay of $\widehat{\phi}$. This proves the assertion.
\end{proof}

\begin{proposition}\label{asymptotic-ft}
Suppose $\phi\in \ms(k)$ and $e(\chi)>-1$. When $|x|$ is sufficiently large, there is
\begin{equation}
\widehat{\phi\chi}(x)=\chi(-1)\phi(0)\gamma(0,\chi^{-1},\psi^{-1})\chi^{-1}(x)|x|^{-1}+
\begin{cases}
0, &\text{$k$ is nonarchimedean,}\\
\mathrm{O}(|x|^{-e(\chi)-2}), &k=\br,\\
\mathrm{O}(|x|^{-e(\chi)-\frac{3}{2}}), &k=\bc.
\end{cases}
\end{equation}
If $k$ is nonarchimedean, the equality holds when $\mathrm{Supp}(\widehat{\phi})\subseteq x\cdot \big(\mathrm{Cond}(\chi)\cap \varpi\mo_k\big)$. If $k$ is archimedean, the equality holds for $|x|\geqs 1$ and the hidden constant in the big $\mathrm{O}$ symbol can be uniform if $\chi$ varies in a compact subset of $\{\chi\in \Hom(k^\times, \bc^\times): e(\chi)>-1\}$.
\end{proposition}
\begin{proof}
When $k$ is nonarchimedean, the assertion follows from (\ref{zeta2}) and Lemma \ref{zetapadic}. When $k$ is archimedean, the assertion follows from (\ref{zeta3}), Lemma \ref{lemmay3}, Lemma \ref{lemmay1}, and the fact that $|\gamma(0,\chi^{-1},\psi^{-1})|$ is uniformly bounded if $\chi$ varies in a compact subset of $\{\chi:e(\chi)>-1\}$.
\end{proof}

\subsubsection{The isometry property of quadratic Fourier transform}\label{qft}

\begin{proposition}\label{isometry-identity}
Suppose $\phi\in \ms(k)$ and $W(x)$ is an even function in $L^1(k)$. If $W(x)|x|^{-1}$, $W^2(x)|x|^{-1}$,, and $\mf_2W(t)|t|^{-1/2}$ are in $L^1(k)$, then
\begin{itemize}
\item[(1)] $\int_k\mf_2W(t)\mf_2\phi(-t)dt=2|2|_k^{-1}\int_k W(x)\phi(x)|x|^{-1}dx$;
\item[(2)] $\int_k\mf_2W(\delta t)\mf_2\phi(-t)dt=0$ for $\delta\in k^\times\backslash {k^\times}^2$.
\end{itemize}
\end{proposition}
\begin{proof}
We may assume that $\phi$ is an even function. Write $\phi=\phi_1+\phi_2$, where $\phi_1=\phi-\phi_2$ and $\phi_2$ is chosen to be (i) $\phi(0)1_{\mo_k}$ if $k$ is nonarchimedean, (ii) $\phi(0)\beta(|x|_k)$ if $k$ is archimedean, where $\beta(u)$ is a non-negative smooth cut-off function on $\br$ taking value $1$ when $|u|\leqs 1$ and value $0$ when $|u|\geqs 2$. We will prove the two assertions in the Proposition for $\phi_i$ separately.

(I) Regarding $\phi_1$, we observe that $\wtilde{\phi}_1|\cdot|_k^{-1/2}\in L^2(k)$. This is because
\begin{itemize}
\item[(a)] If $k$ is nonarchimedean, $\wtilde{\phi}_1$ is locally constant, bounded, and vanishing near $0$.
\item[(b)] If $k$ is archimedean, $\wtilde{\phi}_1$ is bounded and rapidly decreasing. Furthermore, when $|x|_k<1$, there is a constant $c$ such that the value $|\wtilde{\phi}_1(x)|=|\phi(\sqrt{x})-\phi(0)|$ is bounded by $c|x|_\br^{\frac{1}{2}}$ if $k=\br$ or by $c|x|_\bc^{\frac{1}{4}}$ if $k=\bc$, according to mean value theorem.
\end{itemize}
Now the condition $W^2(x)|x|^{-1}\in L^1(k)$ implies $\wtilde{W}|\cdot|^{-\frac{1}{2}}\in L^2(k)$ and $\mf_2W\in L^2(k)$. Applying the $L^2$-isometry property of Fourier transform, we get
\begin{align*}
\int_k\mf_2W(t)\mf_2\phi_1(-t)dt
=&2^2|2|_k^{-2}\int_k\mf(\wtilde{W}|\cdot|^{-1/2})(t)\mf(\wtilde{\phi}_1|\cdot|^{-1/2})(-t)dt\\
=&2^2|2|_k^{-2}\int_k \wtilde{W}(x)|x|^{-1/2}\cdot \wtilde{\phi}_1(x)|x|^{-1/2}dx\\
=&2^2|2|_k^{-2}\int_{k^2} W(\sqrt{x})\phi_1(\sqrt{x})|x|^{-1}dx\\
=&2|2|_k^{-1}\int_k W(x)\phi_1(x)|x|^{-1}dx.
\end{align*}
Similarly,
\[
\int_k\mf_2W(\delta t)\mf_2\phi_1(-t)dt=2^2|2|_k^{-2}\int_k \wtilde{W}(\delta^{-1}x)|\delta x|^{-1/2}\cdot \wtilde{\phi}_1(x)|x|^{-1/2}dx.
\]
If $\delta\in k^\times\backslash {k^\times}^2$, then the right hand side vanishes because $\wtilde{W}(\delta^{-1}x)$ and $\wtilde{\phi}_1(x)$ can not be simultaneously nonzero on $k^\times$.

(II) Regarding $\phi_2$, we may suppose $\phi(0)\neq 0$. Observe that $\wtilde{\phi_2}$ can be written as
\begin{equation}\label{phitwotilded}
\wtilde{\phi}_2=\sum_i c_i\chi_i f,\quad x\in k^\times,
\end{equation}
where $\chi_i$ are quadratic characters of $k^\times$, $f\in C^\infty_c(k)$ is a function taking constant value in a neighborhood of $0$, and $c_i$ are constants. An explicit choice of $f$ is given below:
\begin{itemize}
\item[(a)] When $k$ is nonarchimedean, then $\phi(0)^{-1}\cdot \wtilde{\phi}_2|_{k^\times}$ is the characteristic function of $\mo_k\cap {k^\times}^2$. Let $\{\chi_i\}$ be the finite set of quadratic characters of $k^\times$, then there are constants $c_i$ such that $\wtilde{\phi}_2=\phi(0)\sum_{i}c_i\chi_i1_{\mo_k}$ on $k^\times$. So we may take $f(x)=\phi(0)\cdot 1_{\mo_k}(x)$.

\item[(b)] When $k$ is real, $\wtilde{\phi}_2(x)=\frac{1}{2}(1+\sgn(x))\phi(0)\beta(|x|_\br^{1/2})$ for $x\in \br^\times$. Take $f(x)=\phi(0)\beta(|x|_\br^{1/2})$.

\item[(c)] When $k$ is complex, $\wtilde{\phi}_2(x)=\phi(0)\beta(|x|_\bc^{1/2})$. Take $f(x)=\phi(0)\beta(|x|_\bc^{1/2})$.
\end{itemize}

We claim
\begin{equation}\label{chiequality}
\int_k\mf(\wtilde{W}|\cdot|^{-1/2})(t)\mf(f\chi_i|\cdot|^{-1/2})(-t)dt=\int_k \wtilde{W}(x)f(x)\chi_i(x)|x|^{-1}dx.
\end{equation}
To prove it, consider a family of functions $f\chi_i|\cdot|^{-1/2+s}$, $s\in [0,\frac{1}{4}]$. When $s\in (0,\frac{1}{4}]$, the functions $f\chi_i|\cdot|^{-1/2+s}$ are square-integrable. Combining this with $\wtilde{W}|\cdot|^{-\frac{1}{2}}\in L^2(k)$, we can apply the isometry property of Fourier transform to get
\begin{equation}\label{chiequality2}
\int_k\mf(\wtilde{W}|\cdot|^{-1/2})(t)\mf(f\chi_i|\cdot|^{-1/2+s})(-t)dt=\int_k \wtilde{W}(x)f(x)\chi_i(x)|x|^{-1+s}dx,\quad s\in (0,\frac{1}{4}).
\end{equation}

We would like to let $s\rar 0^+$. The key observations are: (a) $\mf(f\chi_i|\cdot|^{-1/2+s})(t)$ are uniformly bounded and (b) by Proposition ~\ref{asymptotic-ft}, there exist a constant $C$ such that for $s\in [0,\frac{1}{4}]$, there is
\[
\big|\mf(f\chi_i|\cdot|^{-1/2+s})(t)\big|\leqs C|t|^{-\frac{1}{2}-s},\quad |t|\geqs 1.
\]
Combining (a), (b) with the conditions $\mf_2W\in L^2(k)$ and $\mf_2W|\cdot|^{-\frac{1}{2}}\in L^1(k)$, one sees that the family of functions $\mf(\wtilde{W}|\cdot|^{-1/2})\mf(f\chi_i|\cdot|^{-1/2+s})$ on the left hand side of (\ref{chiequality2}) have their modulus dominated by a function in $L^1(k)$.

On the other hand, Combining the fact $f\in C^\infty_c(k)$ and the conditions $W\in L^1(k)$ and $W|\cdot|^{-1}\in L^1(k)$, one sees that the family of functions $\wtilde{W}(x)f(x)\chi_i(x)|x|^{-1+s}$ on the right hand side of (\ref{chiequality2}) also have their modulus dominated by a function in $L^1(k)$. So the dominated convergence theorem is applicable to both sides of (\ref{chiequality2}) when $s\rar 0^+$. This leads to (\ref{chiequality}).

Combining (\ref{chiequality}) and (\ref{phitwotilded}), we obtain
\[
\int_k\mf(\wtilde{W}|\cdot|^{-1/2})(t)\mf(\wtilde{\phi}_2|\cdot|^{-1/2})(-t)dt=\int_k \wtilde{W}(x)\wtilde{\phi}_2(x)|x|^{-1}dx,
\]
which is equivalent to $\int_k\mf_2W(t)\mf_2\phi(-t)dt=2|2|_k^{-1}\int_k W(x)\phi(x)|x|^{-1}dx$.

By the same argument, there is
\[
\int_k\mf_2W(\delta t)\mf_2\phi_2(-t)dt=2^2|2|_k^{-2}\int_k \wtilde{W}(\delta^{-1}x)|\delta x|^{-1/2}\cdot \wtilde{\phi}_2(x)|x|^{-1/2}dx.
\]
If $\delta\in k^\times\backslash {k^\times}^2$, then the right hand side is zero because $\wtilde{W}(\delta^{-1}x)\wtilde{\phi}_2(x)$ vanishes on $k^\times$.
\end{proof}

What we use later is a variation of the above proposition.
\begin{proposition}\label{isometry-identity-extra}
Let $\{\delta_i: 1\leqs i\leqs \ell\}$ be a set of representatives of $k^\times/ {k^\times}^2$ with $\delta_1=1$. Suppose $\phi\in \ms(k)$, $W_1$ is an even function in $L^1(k)$, and $W_i \in L^1(k)$ ($2\leqs i\leqs \ell$). If $W_i(x)|x|^{-1}$, $W_i^2(x)|x|^{-1}$ ($1\leqs i\leqs \ell$), and $\big(\sum_{i=1}^\ell\mf_2W_i(\delta_i t)\big)|t|^{-\frac{1}{2}}$ are in $L^1(k)$, then 
\[
\int_k\Big(\sum_{i=1}^\ell\mf_2 W_i(\delta_i t)\Big)\mf_2\phi(-t)dt=2|2|_k^{-1}\int_k W_1(x)\phi(x)|x|^{-1}dx.
\]
\end{proposition}
\begin{proof}
The argument is the same as for Proposition \ref{isometry-identity}, except that we replace $\mf_2W(t)$ with $\sum_{i=1}^\ell \mf_2W_i(\delta_i t)$ and replace $\wtilde{W}|\cdot|^{-1/2}$ with $\sum_{i=1}^\ell \wtilde{W}(\delta_i^{-1}\cdot)|\delta_i \cdot|^{-1/2}$. So we need the condition $\big(\sum_{i=1}^\ell\mf_2W_i(\delta_i t)\big)|t|^{-\frac{1}{2}}\in L^1(k)$ but do not require $\mf_2W_i(\delta_i t)\big)|t|^{-\frac{1}{2}}\in L^1(k)$ for each $i$.
\end{proof}

\section{The Whittaker functional on $\GL_2$-representations}

In this section, we compare the local inner product with the local Whittaker functional on a local irreducible unitary representation of $\GL_2$, and derive the Whittaker period formula on $\GL_2$.

Let $P$ denote the subgroup of $\GL_2$ consisting of upper triangular matrices, $U$ the unipotent radical of $P$, and $M$ the subgroup of diagonal matrices. Identify $U$ with the affine line $\mathrm{A}^1$ via the map $u=\smalltwomatrix{1}{x}{}{1}\rar x$. Write the Weyl group as $W=\{I_2, w\}$, with $w=\smalltwomatrix{}{1}{-1}{}$.

\subsection{Local theory}\label{sect_gl2local}

Let $k$ be a local field of characteristic zero. Choose a maximal compact subgroup $K$ of $\GL_2(k)$---it can be $\GL_2(\mo_k), \mathrm{O}_2$, or $\mathrm{U}(2)$, depending on $k$. Consider smooth admissible representations of $\GL_2(k)$. If $k$ is archimedean, these are understood as admissible $(\fg, K)$-modules; by a theorem of Casselman-Wallach, each $(\fg,K)$-module $\pi$ of finite length can be uniquely globalized to a smooth admissible representation $\pi^\infty$ of $\GL_2(k)$ on a Fr\'{e}chet space. If $k$ is nonarchimedean, the underlying space is topologized with the trivial locally convex topology in which every semi-norm is continuous (See \cite[Section 2]{jl70}).

Let $\msr$ denote the set of irreducible smooth admissible representations of $\GL_2(k)$ and $\msr_u$ the subset of unitarizable ones. Members of $\msr_u$ are of the following form (See \cite[Theorem 12]{gode70} for the nonarchimedean case):
\begin{itemize}
\item[(a)] an induced representation $\pi(\mu_1,\mu_2)=\Ind_P^{\GL_2}(\mu_1,\mu_2)$, where $\mu_1, \mu_2$ are unitary or $\mu_1=\mu|\cdot|^\alpha$, $\mu_2=\mu|\cdot|^{-\alpha}$ with $\mu$ unitary and $\alpha\in (-\frac{1}{2},0)\cup(0,\frac{1}{2})$,
\item[(b)] a one-dimensional representation $\mu(\det g)$ with $\mu$ unitary,
\item[(c)] the Steinberg representation $\mathrm{St}(\mu)$ with $\mu$ unitary, which is the unique subrepresentation of $\Ind_P^{\GL_2}(\mu|\cdot|^{\frac{1}{2}},\mu|\cdot|^{-\frac{1}{2}})$, if $k$ is nonarchimedean; or the representation $\mu\otimes D_p$ ($p\geqs 2$), where $\mu$ is unitary and $D_p$ is the unique subrepresentation of $\Ind(|\cdot|^{\frac{p-1}{2}}\sgn^{p},|\cdot|^{-\frac{p-1}{2}})$, if $k=\br$,
\item[(d)] a supercuspidal representation with unitary central character, if $k$ is nonarchimedean.
\end{itemize}
Note that representations of type (c) and (d) are square-integrable. Representations of type (c), (d), and type (a) with $\mu_1,\mu_2$ unitary are tempered.\\

\indent Choose a nontrivial $\psi\in \Hom(k,\bc^\times)$. A Whittaker functional on a smooth admissible representation $\pi$ with respect to $\psi$ is an element in $\Hom_{U(k)}(\pi,\psi)$, that is, a continuous linear functional $\ell_\psi:\pi\rar \bc$ satisfying
\[
\ell_\psi\big(\pi(u)f\big)=\psi(u)\ell_\psi(f),\quad u\in U(k), f\in \pi.
\]
If $\pi\in \msr$ is infinite dimensional, then $\dim\Hom_{U(k)}(\sigma,\psi)=1$ by \cite{jl70}; we choose a nonzero Whittaker functional $\ell_\psi$ and associate to each $f\in \pi$ a Whittaker function $W_{f,\psi}(g)=\ell_\psi(\pi(g)f)$, $g\in \GL_2(k)$. (Note: when $k$ is archimedean, $\ell_\psi$ is a functional on $\pi^\infty$.)

\subsubsection{Estimate of matrix coefficients and Whittaker functions} Suppose $\pi\in \msr_u$ is infinite dimensional. We define a number $\alpha(\pi)\in \br$ by
\begin{equation*}
\alpha(\pi)=\begin{cases}
0, &\text{if $\pi$ is tempered},\\
|\alpha|, &\text{if $\pi=\Ind_P^{\GL_2}(\mu|\cdot|^{\alpha},\mu|\cdot|^{-\alpha})$.}
\end{cases}
\end{equation*}
Then $\alpha(\pi)<\frac{1}{2}$, due to the description of $\msr_u$. The following asymptotic estimates hold. For square-integrable representations, they can be sharpened though the sharpened one is not needed.

\begin{lemma}\label{gl2estimate}
Suppose $\pi\in \msr_u$ is infinite dimensional and $f_1, f_2, f\in \pi$.\\
\emph{(i)} For any $\epsilon>0$, there is a constant $C_{\epsilon,f_1,f_2}$ such that
\begin{align}
\big|\big(\pi\smalltwomatrix{a}{}{}{1}f_1,f_2\big)\big|&\leqs C_{\epsilon,f_1,f_2} |a|^{\frac{1}{2}-\alpha(\pi)-\epsilon},\quad\quad |a|\leqs 1, \label{abound}\\
\big|\big(\pi (w\smalltwomatrix{1}{x}{}{1})f_1,f_2\big)\big|&\leqs C_{\epsilon, f_1,f_2}|x|^{-1+2\alpha(\pi)+\epsilon},\quad |x|\geqs 1. \label{ubound}
\end{align}
\emph{(ii)} When $a$ is near infinity, $W_{f,\psi}\smalltwomatrix{a}{}{}{1}$ vanishes if $k$ is nonarchimdean and rapidly decays if $k$ is archimedean. For any $\epsilon>0$, there is a constant $C_{\epsilon,f}$ such that
\begin{equation}
\big|W_{f,\psi}\smalltwomatrix{a}{}{}{1}\big|\leqs C_{\epsilon,f}|a|^{\frac{1}{2}-\alpha(\pi)-\epsilon},\quad\quad\quad\quad\quad |a|\leqs 1. \label{wbound}
\end{equation}
\end{lemma}
\begin{proof}
In part (i), (\ref{ubound}) is a consequence of (\ref{abound}), by the KAK decomposition and the $K$-finiteness of $\varphi_1,\varphi_2$. Estimate (\ref{abound}) is well-known. (For example, it can be verfied with the same argument as in the proof of Lemma ~\ref{sl2mbound} and we skip it for brevity.)

Part (ii) follows from the explicit study of the Whittaker functional in \cite{jl70, gode70}, see \cite[section 1.8 \& 1.10]{gode70} when $k$ is nonarchimdean and \cite[section 2.5, Equations (68) \& (84)]{gode70} when $k$ is archimedean. If $k$ is archimedean, it alternatively follows from the general estimate in \cite[section 15.2.2]{wallach92}.
\end{proof}

\subsubsection{Local duality}

Suppose $\pi\in \msr_u$ is infinite dimensional. Let $(,)$ be an inner product on $\pi$ and $\ell_\psi$ be a nonzero Whittaker functional on $\pi$. We explicate the relation between $(,)$ and $\ell_\psi$.

First, we regularize the integral $\int_k (\pi(u)f_1,f_2)\psi(-u)du$. Let $T_{(\pi(u)f_1,f_2)}$ be the linear functional defined by $T_{(\pi(u)f_1,f_2)}\lambda=\int_k (\pi(u)f_1,f_2)\lambda(u)du$, $\lambda\in \ms(k)$; it is in $\ms^\prime(k)$ and, as shown in Remark ~\ref{gl2regular} below, its Fourier transform $\mf T_{(\pi(u)f_1,f_2)}$ is represented over $k^\times$ by a smooth function $W_{f_1,f_2,\psi}(a)$.
\begin{definition}\label{defgl2re}
Define
\begin{equation}\label{regularizedi}
\int_k (\pi(u)f_1,f_2)\psi(-u)du\overset{\triangle}{=}W_{f_1,f_2,\psi}(-1).
\end{equation}
Obviously, $W_{f_1,f_2,\psi}(-1)\in \Hom_{U(k)}(\pi,\psi)\otimes \overline{\Hom_{U(k)}(\pi,\psi)}$.
\end{definition}

\begin{remark}\label{gl2regular}
We briefly argue for the smoothness of $W_{f_1,f_2,\psi}(t)$ on $k^\times$.
\begin{itemize}
\item[(i)] When $k$ is nonarchimedean, there exists a compact open subgroup $U_0$ of $\mo_k^\times$ such that $(\pi(ua)f_1,f_2)$ is independent of $a$ when $a\in U_0$. It follows that for $\lambda(t)\in \ms(k^\times)$, there is
\begin{align*}
\int_k (\pi(u)f_1,f_2)\widehat{\lambda}(u)du=&\int_k\int_{U_0}(\pi(au)f_1,f_2)\widehat{\lambda}(u)d^\times adu\\
=&\int_k \lambda(t)\Big(\int_k \mathrm{Vol}(U_0)^{-1}\cdot \widehat{1_{U_0}}(tu)(\pi(u)f_1,f_2)du\Big)dt
\end{align*}
Hence $W_{f_1,f_2,\psi}(t)=\int_k \mathrm{Vol}(U_0)^{-1}\widehat{1_{U_0}}(tu)(\pi(u)f_1,f_2)du$ is a $U_0$-invariant function on $k^\times$.
\item[(ii)] When $k$ is archimedean, by Dixmier-Malliavin Theorem, there exists finitely many smooth vectors $f_{1,i}, f_{2,i}$ in the unitary closure of $\pi$ and functions $\lambda_i(a)\in C^\infty_c(F_v^\times)$ such that
\[
f_1\otimes f_2=\sum_i \int_{k^\times} \lambda_i(a)\cdot \pi\smalltwomatrix{a}{}{}{1}f_{1,i}\otimes \pi\smalltwomatrix{a}{}{}{1}f_{2,i} d^\times a.
\]
Hence $(\pi(u)f_1,f_2)=\sum_i \int_{k^\times}\lambda_i(a) (\pi(a^{-1}u)f_{1,i},f_{2,i})d^\times a$. For $\lambda(t)\in C^\infty_c(k^\times)$, this implies
\[
\int_k (\pi(u)f_1,f_2)\widehat{\lambda}(u)du=\int_k \lambda(t)\Big(\sum_i \int_k \widehat{\lambda_i}(tu)\big(\pi(u)f_{1,i},f_{2,i}\big)du\Big)dt
\]
Hence $W_{f_1,f_2,\psi}(t)=\sum_i\int_k \widehat{\lambda_i}(tu)\big(\pi(u)f_{1,i},f_{2,i}\big)du$ is a smooth function on $k^\times$.
\end{itemize}
\end{remark}

\begin{remark}\label{gl2wcomput}
$W_{f_1,f_2,\psi}(a)$ can be computed by employing a sequence of functions that ``converge" to the single point $a$. Precisely, let $\phi_n\in \ms(k)$ ($n\geqs 1$) be such that $T_{\phi_n}\rar \delta_{a}$ in $\ms^\prime(k)$, then
\begin{equation}\label{abelintegral}
W_{\psi,f_1,f_2}(a)=\lim_{n\rar \infty} \int_k W_{\psi,f_1,f_2}(x)\phi_n(x)da=\lim_{n\rar \infty}\int_k \big(\pi(u)f_1,f_2)\hat{\phi}_n(u)du.
\end{equation}
If $k$ is nonarchimedean, take $\phi_n(x)=|\varpi|^{-n}1_{\varpi^n\mo_k}(x+a)$, then
\begin{equation}\label{gl2stable}
W_{f_1,f_2,\psi}(-a)=\lim_{n\rar \infty} \int_{\varpi^{-n}\mo_k} (\pi(u)f_1,f_2)\psi(-au)du.
\end{equation}
Note that the integral on the right hand side of (\ref{gl2stable}) becomes stable for sufficiently large $n$.
\end{remark}

Second, we relate the local inner product and the local Whittaker functional.
\begin{lemma}\label{gl2localduality}
There exists a nonzero constant $c$ such that for all $f_1, f_2\in \pi$,
\begin{align}
\ell_\psi(f_1)\overline{\ell_\psi(f_2)}=&\frac{1}{c}\int_k (\pi(u)f_1,f_2)\psi(-u)du, \label{m2wgl2}\\
(f_1,f_2)=&c\int_{k^\times} W_{f_1,\psi}\overline{W_{f_2,\psi}}\big(\smalltwomatrix{a}{}{}{1}\big)d^\times a. \label{w2mgl2}
\end{align}
\end{lemma}
\begin{proof}
Note that the right hand side of (\ref{w2mgl2}) is convergent by Lemma \ref{gl2estimate} (ii).

We first prove (\ref{m2wgl2}). Because $\dim\Hom_{U(k)} (\pi,\psi)=1$, there is a constant $c$ such that $W_{f_1, f_2,\psi}(-1)=c\cdot \ell_\psi(f_1)\overline{\ell_\psi(f_2)}$ for all $f_1,f_2\in \pi$. This implies 
\begin{equation}\label{mwrelationgl2}
W_{f_1,f_2,\psi}(-a)=c|a|^{-1} \ell_\psi\big(\pi\smalltwomatrix{a}{}{}{1}f_1\big)\overline{\ell_\psi\big(\pi\smalltwomatrix{a}{}{}{1}f_2\big)},\quad \forall a\in k^\times.
\end{equation}
Actually, choose a sequence of functions $\phi_n\in \ms(k)$ such that $T_{\phi_n}\rar \delta_1$ in $\ms^\prime(k)$, then $T_{\phi_{n,a}}\rar \delta_a$ with $\phi_{n,a}(x)=\phi_n(ax)$. Applying formula (\ref{abelintegral}) to $\phi_{n,a}$ yields (\ref{mwrelationgl2}). Now for (\ref{m2wgl2}), it suffices to show the constant $c$ is nonzero. Actually, if $c=0$, then $W_{f_1,f_2,\psi}(a)=0$ on $k^\times$, whence $\mf T_{(\pi(u)f_1,f_2)}$ is a tempered distribution on $k$ supported at $\{0\}$. By Lemma ~\ref{pointsupport}, $\mf T_{(\pi(u)f_1,f_2)}$ is a multiple of $\delta_0$ if $k$ is nonarchimedean or a finite linear combination of derivatives of $\delta_0$ if $k$ is archimedean. Accordingly, $(\pi(u)f_1,f_2)$ is a constant function if $k$ is nonarchimedean or a polynomial function if $k$ is archimedean. Since $(\pi(u)f_1,f_2)$ decays to $0$ at infinity by (\ref{ubound}), it needs to be zero, but this is a contradiction when $u=1$ and $f_1=f_2\neq 0$. Therefore, $c$ must be nonzero.

For (\ref{w2mgl2}), set $M_{f_1,f_2}(u)=\int_{k^\times} cW_{f_1,\psi}\overline{W_{f_2,\psi}}\big(\smalltwomatrix{a}{}{}{1}\big)\psi(au)d^\times a$, $u\in k$. Then $\mf T_{M_{f_1,f_2}(u)}$ and $\mf T_{(\pi(u)f_1,f_2)}$ agree on $k^\times$ and their difference is a tempered distribution supported at $\{0\}$. By the same argument as before, $M_{f_1,f_2}(u)-(\pi(u)f_1,f_2)$ is a constant function if $k$ is nonarchimedean or a polynomial function if $k$ is archimedean. Because $(\pi(u)f_1,f_2)$ decays to $0$ and $M_{f_1,f_2}(u)$ also decays to $0$ (as the Fourier transform of an $L^1$-function), one must have $M_{f_1,f_2}(u)-(\pi(u)f_1,f_2)=0$.
\end{proof}

\subsection{Global theory}\label{gl2_pwf}

Fix a non-trivial character $\psi$ of $\ba/F$. Let $\pi$ be an irreducible unitary cuspidal automorphic representation of $\GL_2(\ba)$. Equip $\PGL_2(\ba)$ and $U(\ba)=\ba$ with the Tamagawa measures. The standard inner product and the standard Whittaker functional on $\pi$ are :
\begin{align*}
&(f_1,f_2)=\int_{\PGL_2(F)\backslash \PGL_2(\ba)} f_1(h)\overline{f_2}(h)dh,\quad  \ell_\psi(f)=\int_{U(F)\backslash U(\ba)}f(u)\psi^{-1}(u)du.
\end{align*}
We shall derive the relation between $(,)$ and $\ell_\psi$. (Note that $\mathrm{Vol}(\PGL_2(\bq)\backslash \PGL_2(\ba))=2$.)

We fix a maximal compact subgroup $K=\prod_v K_v$ of $\GL_2(\ba)$, with $K_v=\GL_2(\mo_{F_v}), \OO_2(\br)$, or $\mathrm{U}_2$. Write $\pi=\otimes_v\pi_v$ as the restricted tensor product of irreducible admissible unitarizable representations $\pi_v$ of $\GL_2(F_v)$, where almost all $\pi_v$ are spherical with respect to $K_v$. Choose a spherical vector $f_{v,0}$ in each spherical $\pi_v$, then $\pi$ is spanned by decomposable vectors of the form $f=\otimes f_v$, where $f_v=f_{v,0}$ at almost all places. At each place $v$, We fix a choice of non-zero Whittaker functional $\ell_{\psi_v}$ on $\pi_v$, requiring that $\ell_{\psi_v}(f_{v,0})=1$ for almost all spherical $\pi_v$ and that $\ell_\psi=\prod_v \ell_{\psi_v}$; we also fix a choice of local inner product $(,)_v$ on $\pi_v\otimes \pi_v$, requiring that $(f_{v,0},f_{v,0})=1$ for almost all spherical $\pi_v$ and that $(,)=\prod_v (,)_v$. Take the Tamagawa measures $da=\prod_v da_v$ on $\ba$ and $d^\ast a=\prod_v da^\ast_v$ on $\ba^\times$. Recall that $d^\ast a_v=\zeta_{F_v}(1)d^\times a_v$ for almost all $v$.

At each place $v$ of $F$, according to Lemma ~\ref{gl2localduality}, there is a local constant $c_v$ such that
\begin{align}
(f_{1,v},f_{2,v})_v=&c_v\cdot \int_{F_v^\times}W_{f_{1,v},\psi_v}\overline{W_{f_{2,v},\psi_v}}\smalltwomatrix{a_v}{}{}{1}d^\times a_v, \label{gl2vdual1}\\
\ell_{\psi_v}(f_{1,v})\overline{\ell_{\psi_v}(f_{2,v})}=&\frac{1}{c_v}\cdot\int_k (\pi_v(u_v)f_{1,v},f_{2,v})\psi(-u_v)du_v,\quad f_{1,v}, f_{2,v}\in \pi_v. \label{gl2vdual2}
\end{align}
It is easy to see $c_v=\frac{\zeta_{F_v}(2)}{L(1,\pi_v,\ad)}$ for almost all $v$, by either applying the Macdonald formula of spherical matrix coefficients to (\ref{gl2vdual2}) or the Casselman-Shalika formula of spherical Whittaker functions to (\ref{gl2vdual1}). Thus, we define the according normalized pairings below, with which one can immediately write down the Whittaker period formula up to a global constant.
\begin{align}
(f_{1,v},f_{2,v})_v^\sharp\, &\overset{\triangle}{=}\frac{\zeta_{F_v}(2)}{L(1,\pi_v,\ad)\zeta_{F_v}(1)}\int_{F_v^\times}W_{f_{1,v},\psi_v}\overline{W_{f_{2,v},\psi_v}}[\smalltwomatrix{a_v}{}{}{1}]d^\ast a_v,\label{nfunctional11}\\
\ml_\psi^\sharp (f_{1,v },f_{2,v})\, &\overset{\triangle}{=}\frac{L(1,\pi_v,\ad)}{\zeta_{F_v}(2)} \int_{U_v}(\pi_v(u_v)f_{1,v},f_{2,v})_v\psi_v(-u_v)du_v. \label{nfunctional12}
\end{align}
\begin{lemma}
Set $c_\pi=\frac{L(1,\pi,\ad)}{\zeta_F(2)} \prod_v \frac{\zeta_{F_v}(2)}{c_vL(1,\pi_v,\ad)}$, then
\begin{align}
&(f_1,f_2)=\frac{1}{c_\pi}\cdot \frac{L(1,\pi,\ad)\Res_{s=1}\zeta_F(s)}{\zeta_F(2)}\prod_v (f_{1,v},f_{2,v})_v^\sharp, \label{gl2dual1}\\
&\ell_\psi(f_1)\overline{\ell_\psi(f_2)}=c_\pi \cdot \frac{\zeta_F(2)}{L(1,\pi,\ad)}\prod_v \ml_\psi^\sharp (f_{1,v },f_{2,v}). \label{gl2dual2}
\end{align}
\end{lemma}

We compute the constant $c_\pi$ below with the Rankin-Selberg integral on $\GL_2\times \GL_2$. In \cite{lapid_mao2015}, the Rankin-Selberg integral has been used to deduce a similar Whittaker period formula on $\GL_m$. The difference between our approach and that of \cite{lapid_mao2015} lies in the regularization of the local integral of matrix coefficients. When $k$ is nonarchimedean, the regularization used in \cite{lapid_mao2015} is based on stable integrals and agrees with our regularization in terms of Fourier transform and distribution (c.f. Definition ~\ref{defgl2re} and Remark ~\ref{gl2wcomput}); one may view stable integrals as a way to compute the smooth function representing the target distribution with certain characteristic functions as test functions. When $k$ is archimedean, the regularization in \cite{lapid_mao2015} is an ad hoc one based on the explicit structure of the local representation as a subquotient of an induced representation and hence is not suitable for the comparison of local periods. Our regularization is of function-theoretic nature and uniform at every local place, thus allowing for a natural comparison of local Whittaker periods.

We also note that the regularization in Definition ~\ref{defgl2re}, which works for all infinite dimensional iireducible unitary representations of $\GL_2(k)$, can be generalized to local Whittaker period integrals on $\GL_m(k)$ ($m\geqs 3$) for tempered representations. For the sake of brevity, we do not present the generalized regularization. A similar mechanism for local Gross-Prasad period integrals has been demonstrated in \cite[Section 3.4]{liu2016} for tempered representations at archimedean places and would work at nonarchimedean places as well; the key estimate in \cite[Proposition 3.10]{liu2016} was first (essentially) proved by Waldspurger \cite[Section 4]{wald2012} at nonarchimedean places and then extended by Liu to archimedean places.

\begin{proposition}\label{compare_gl2ip}
$c_\pi=\frac{1}{2}$.
\end{proposition}
\begin{proof}
Let $Z$ be the center of $\GL_2$ and put $\overline{P}=Z\backslash P$, $\overline{K}=Z(\ba)\backslash KZ(\ba)$. Set $\mathrm{Vol}(\overline{K})=1$, then the Tamagawa measure on $\PGL_2(\ba)$ can be decomposed as $dh=c|a|^{-1}dx d^\ast a d\bar{k}$ for $h=\smalltwomatrix{1}{x}{}{1}\smalltwomatrix{a}{}{}{1}\bar{k}$, where $dx, d^\ast a$ are Tamagawa measures on $\ba, \ba^\times$ respectively and $c$ is a constant.

Let $\Ind_{\overline{P}}^{\PGL_2}|\cdot|^s$ be the representation of $\PGL_2(\ba)$ (unitarily) induced from the quasi-character $\smalltwomatrix{a_1}{x}{}{a_2}\rar |\frac{a_1}{a_2}|^s$. Let $F_s$ be the spherical section satisfying $F_s|_{\overline{K}}=1$. The Eisenstein series
\[
E(F_s)(h)=\sum_{\gamma\in P(F)\backslash \GL_2(F)}F_s(\gamma h)
\]
is absolutely convergent when $\re(s)>\frac{1}{2}$ and has meromorphic continuation to whole complex $s$-plane;  it has a simple pole at $s=\frac{1}{2}$ with constant residue $\kappa$. Hence
\begin{align*}
(f_1,f_2)
=&\frac{1}{\kappa}\int_{[\PGL_2]}f_1(h)\overline{f_2}(h)\big[\Res_{s=\frac{1}{2}}E(F_s)\big](h)dh\\
=&\frac{1}{\kappa}\Res_{s=\frac{1}{2}}\int_{[\PGL_2]}f_1(h)\overline{f_2}(h)E(F_s)(h)dh.
\end{align*}

Using the Whittaker expansion of $f_1,f_2$ and doing unfolding-folding, one can write
\begin{align*}
\int_{[\PGL_2]}f_1(h)\overline{f_2}(h)E(F_s)(h)dh
=&\int_{\overline{P}(\ba)\backslash \PGL_2(\ba)}\int_{\ba^\times}W_{f_1,\psi}\overline{W_{f_2,\psi}}\big(\smalltwomatrix{a}{}{}{1}\dot{h}\big)F_s(\dot{h})|a|^{s-\frac{1}{2}}d^\ast a d\dot{h}\\
=&c\int_{\overline{K}}\int_{\ba^\times}W_{f_1,\psi}\overline{W_{f_2,\psi}}\big(\smalltwomatrix{a}{}{}{1}\bar{k}\big)|a|^{s-\frac{1}{2}}d^\ast a d\bar{k}\\
=&c\prod_v \int_{\overline{K_v}}\int_{F_v^\times}W_{f_{1,v},\psi_v}\overline{W_{f_{2,v},\psi_v}}\big(\smalltwomatrix{a_v}{}{}{1}\bar{k}_v\big)|a_v|^{s-\frac{1}{2}}d^\ast a_v d\bar{k_v}.
\end{align*}

We observe that the local integral $\int_{\overline{K_v}}\int_{F_v^\times}$ on the right hand side equals $\frac{L(s,\pi_v,\ad)\zeta_{F_v}(s+\frac{1}{2})}{\zeta_{F_v}(2s+1)}$ at almost all places. Thus, one can rewrite
\[
(f_1,f_2)=\frac{c}{\kappa}\Res_{s=\frac{1}{2}}\left[\frac{L(s+\frac{1}{2},\pi,\ad)\zeta_F(s+\frac{1}{2})}{\zeta_F(2s+1)}\right]\cdot \prod_v (f_{1,v},f_{2,v})^\prime,\]
with
\begin{align*}
(f_{1,v},f_{2,v})^\prime
:=&\left[\frac{\zeta_{F_v}(2s+1)\int_{\overline{K_v}}\int_{F_v^\times}W_{f_{1,v},\psi_v}\overline{W_{f_{2,v},\psi_v}}\big(\smalltwomatrix{a_v}{}{}{1}\bar{k}_v\big)|a_v|^{s-\frac{1}{2}}d^\ast a_v d\bar{k}_v}{L(s+\frac{1}{2},\pi_v,\ad)\zeta_{F_v}(s+\frac{1}{2})}\right]_{s=\frac{1}{2}}\\
=&\frac{\zeta_{F_v}(2)}{L(1,\pi_v,\ad)\zeta_{F_v}(1)}\int_{\overline{K_v}}\int_{F_v^\times}W_{f_{1,v},\psi_v}\overline{W_{f_{2,v},\psi_v}}\big(\smalltwomatrix{a_v}{}{}{1}\bar{k}_v\big)d^\ast a_v d\bar{k}_v\\
=&\int_{\overline{K_v}}(\bar{k}_v\circ f_{1,v},\bar{k}_v\circ f_{2,v})_v^\sharp d\bar{k}_v\\
=&(f_{1,v},f_{2,v})_v^\sharp.
\end{align*}
Note that we have used the fact that $(\cdot,\cdot)_v^\sharp$ is $K_v$-invariant. It follows that
\[
(f_1,f_2)=\frac{c}{\kappa}\cdot \frac{L(1,\pi,\ad)\Res_{s=1}\zeta_F(s)}{\zeta_F(2)}\prod_v (f_{1,v},f_{2,v})^\sharp.
\]
So $c_\pi=\kappa/c$. It is known that $c=2\kappa$, whence $c_\pi=\frac{1}{2}$. (See \cite[Lemma 3.4(ii)]{gqt14} for the general relation between the measure constant in Iwasawa decomposition and the residue of spherical Eisenstein series.)
\end{proof}

\section{The Whittaker functional on $\wtilde{\SL}_2$-representations}\label{sl2context}

In this section, we (1) compare the local inner product and the local Whittaker functional on an local irreducible unitary representation of $\wtilde{\SL}_2$, and, (2) for a cuspidal automorphic representation of $\wtilde{\SL}_2$ that is orthogonal to elementary theta series, write down its Whittaker period formula with the global constant therein expressed as a product of local constants.

Let $B$ be the subgroup of $\SL_2$ consisting of upper triangular matrices, $N$ the unipotent radical of $B$, and $A$ the diagonal subgroup of $B$. Identify $N$ with the affine line and keep $w=\smalltwomatrix{}{1}{-1}{}$.

\subsection{Local Theory}\label{sl2local}

Let $k$ be a local field of characteristic zero. Fix a nontrivial $\psi\in \Hom(F_v,\mathrm{S}^1)$.

\subsubsection{$\wtilde{\SL}_2(k)$ as a topological group}

The group $\wtilde{\SL}_2(k)$ is the non-trivial two-fold cover of $\SL_2(k)$ if $k\neq \bc$ and the trivial two-fold cover of $\SL_2(k)$ if $k=\bc$. For $X\subseteq \SL_2(k)$, denote its preimage in $\wtilde{\SL}_2(k)$ by $\wtilde{X}$. Choose a maximal compact subgroup $K$ of $\SL_2(k)$, which is $\SL_2(\mo_k)$, $\SO_2$, or $\SU_2$.
\begin{itemize}
\item[(i)] Identify $\wtilde{\SL}_2(k)$ with $\SL_2(k)\times \{\pm 1\}$ as sets, then the group law is
\[
[g_1,\epsilon_1][g_2,\epsilon_2]=[g_1\cdot g_2, \epsilon(g_1,g_2)\epsilon_1\epsilon_2],
\]
where $\epsilon(g_1,g_2)=<j(g_1)j(g_1g_2)|j(g_2)j(g_1g_2)>$ and the $j$ function therein is
\begin{equation*}
j\smalltwomatrix{a}{b}{c}{d}=
\begin{cases}
c,  &\text{if}\quad c\neq 0\\
a,  &\text{if} \quad c=0.
\end{cases}
\end{equation*}

\item[(ii)] If $k$ is nonarchimedean or $\bc$, the topology on $\wtilde{\SL}_2(k)$ is the product topology; if $k=\br$, $\wtilde{\SL}_2(k)$ is homeomorphic to $B(\br)\times \big(\SO_2\times\bz/2\bz\big)$, with $\SO_2\times\bz/2\bz$ identified with $\bz/4\pi\bz$:
\begin{align*}
  \wtilde{\gamma}:\br/4\pi\bz &\isoto \SO_2(\br)\times \bz/2\bz,\\
    [\theta] &\lrar [\smalltwomatrix{\cos \theta}{\sin \theta}{-\sin \theta}{\cos \theta}
    ,1_{(-\pi,\pi]+4\pi\bz}(\theta)-1_{(\pi,3\pi]+4\pi\bz}(\theta)].
\end{align*}

\end{itemize}
Note that $\wtilde{N}(k)\rar N(k)$ is always split, $\wtilde{A}(k)\rar A(k)$ is split only when $k=\bc$, and $\wtilde{K}\rar K$ is split when $k$ is neither real nor dyadic. We use the same symbol $N, K$ for the lifts if they exist.

For $a\in k^\times$ write $\underline{a}$ for $[\smalltwomatrix{a}{}{}{a^{-1}},1]$. For $n\in k$, use the same notation $n$ for the element $[\smalltwomatrix{1}{n}{}{1},1]$. In general, for $g$ of $\SL_2(k)$, it is understood to be $[g,1]$ when regarded as an element of $\wtilde{\SL}_2(k)$.

\subsubsection{The Weil representation}\label{localweilrep}

Let $(V,q)$ be a quadratic space over $k$ of dimension $m$ and write $q(x,y)=q(x+y)-q(x)-q(y)$. The Weil representation $\omega_{\psi,V}$ of $\wtilde{\SL}_2(k)\times \OO(V)$ on $\ms(V)$ is given by the formulas below: $a\in k^\times, u\in k$, $h\in \OO(V)$, $\phi\in \ms(V)$,
\begin{align*}
\omega_{\psi,V}\big[\smalltwomatrix{a}{}{}{a^{-1}},\epsilon\big]\phi(x)=&\epsilon^m\chi_{\psi,V}(a)|a|^{\frac{m}{2}}\phi(ax),\\
\omega_\psi\big[\smalltwomatrix{1}{n}{}{1},1\big]\phi(x)=&\psi\big(q(x)n\big)\phi(x),\\
\omega_{\psi,V}[w,1]\phi(x)=&\gamma(\psi,V) \int_k \psi\big(q(x,y)\big)\phi(y)dx,\\
\omega_{\psi,V}(h)\phi(x)=&\phi(h^{-1}x),
\end{align*}
Here $\big[\smalltwomatrix{a}{}{}{a^{-1}},\epsilon\big]\rar \epsilon^m\chi_{\psi,V}(a)$ is a character on $\wtilde{A}(k)$ and $\gamma(\psi,V)$ is a number of norm $1$.

When $\dim V=1$ and $q(x)=x^2$, we write simply $\omega_\psi$, $\chi_{\psi}$, $\gamma_\psi$ for the according notations $\omega_{\psi,V}, \chi_{\psi,V}$ and $\gamma(\psi,V)$. The number $\gamma_\psi$ is an eighth-root of unity and there is
\[
\chi_\psi(a)=<a,-1>\cdot \frac{\gamma_{\psi_a}}{\gamma_\psi}.
\]
The function $\chi_\psi(a)$ satisfies $\chi_\psi(a_1a_2)=<a_1,a_2>\chi_\psi(a_1)\chi_\psi(a_2)$, $\chi_\psi(a^2)=1$, and $\chi_{\psi_a}=\chi_\psi\cdot\chi_a$. For a general $V$, write $q(x)=a_1x_1^2+\cdots +a_mx_m^2$, then $\gamma(\psi,V)=\prod_j \gamma_{\psi_{a_j}}$ and $\chi_{\psi,V}=\prod_j \chi_{\psi_{a_j}}$.

Specifically, $\omega_\psi=\omega_{\psi}^+\oplus \omega_{\psi}^-$, where $\omega^+_{\psi}$ (resp. $\omega^-_{\psi}$) is the action of $\wtilde{\SL}_2(k)$ on the subspace of $\ms(k)$ consisting of even (resp. odd) functions. Both $\omega_\psi^+$ and $\omega_\psi^-$ are irreducible and they are called the even and odd Weil representations of $\wtilde{\SL}_2(k)$. When $k$ is archimedean, they are irreducible smooth representations and, for simplicity, we use the same notation for their associated $(\fg,\wtilde{K})$-modules.

Note that $\omega_{\psi_{a_1}}$ and $\omega_{\psi_{a_2}}$ are equivalent if and only if $a_1a_2^{-1}\in {k^\times}^2$. For a quadratic character $\chi$ of $k^\times$, we write $\omega_{\psi,\chi}^\pm=\omega_{\psi,\psi_a}^\pm$, where $a$ is any number in $k^\times$ such that $\chi_a=\chi$.

\subsubsection{The unitary dual of $\wtilde{\SL}_2(k)$}

We consider genuine smooth admissible representations of $\wtilde{\SL}_2(k)$ with respect to $\wtilde{K}$. If $k$ is archimedean, these are understood as admissible $(\fs\fl_2, \wtilde{K})$-modules and each such $\sigma$ of finite length can be uniquely globalized to a smooth admissible representation $\sigma^\infty$ on a Fr\'{e}chet space; we will be more specific when $\sigma$ and $\sigma^\infty$ need to be distinguished. If $k$ is nonarchimedean, the underlying space is topologized as in section ~\ref{sect_gl2local}.

Let $\tildesr$ denote the set of irreducible genuine smooth admissible representations of $\wtilde{\SL}_2(k)$ and $\tildesr_u$ the subset of unitarizable ones. The structure of $\tildesr_u$ follows from that of $\msr_u$ when $k=\bc$, is worked out in \cite[Lemma 4.2]{gelbart76} when $k=\br$, and is described in \cite{wald80} when $k$ is nonarchimedean. According to these sources, a member of $\tildesr_u$ is of the following form:
\begin{itemize}
\item[(a)] an induced representation $\sigma(\mu)=\Ind_{\wtilde{B}}^{\wtilde{\SL}_2}(\mu\chi_\psi)$, where $\mu$ is unitary or $\mu=\chi|\cdot|^\alpha$ with $\chi$ quadratic and $\alpha\in (-\frac{1}{2},0)\cup (0,\frac{1}{2})$;
\item[(b)] an even Weil representation $\omega_{\psi,\chi}^+$ with $\chi$ quadratic, if $k$ is nonarchimedean or real, or a one-dimensional unitary representation, if $k$ is complex;
\item[(c)] the Steinberg representation $\mathrm{St}_{\psi}(\chi)$ with $\chi$ quadratic, if $k$ is nonarchimedean; or the weight-($\ell+\frac{1}{2}$) discrete series representations $\wtilde{D}_{\ell+\frac{1}{2}}^{\epsilon}$ with $\ell\in\bn$ and $\epsilon\in \{\pm 1\}$, if $k=\br$;
 
\item[(d)] a supercuspidal representation, if $k$ is nonarchimedean.
\end{itemize}
Note that representations of type (c) and (d) are square-integrable. They and type (a) representations with $\mu$ unitary are tempered. We make a few comments below.

First, for $\mu\in \Hom(k^\times,\bc^\times)$, the induced representaiton $\Ind_{\wtilde{B}}^{\wtilde{\SL}_2}(\mu\chi_\psi)$ is the right translation action of $\wtilde{\SL}_2(k)$ on the space of $\wtilde{K}$-finite functions $\varphi:\wtilde{\SL}_2(k)\rar \bc$ satisfying
\[
  \varphi\big([\smalltwomatrix{a}{x}{}{a^{-1}},\epsilon]g\big)=\epsilon\chi_{\psi}(a)\mu(a)|a| \varphi(g).
\]

(i) The dual of $\Ind_{\wtilde{B}}^{\wtilde{\SL}_2}(\mu\chi_\psi)$ is $\Ind_{\wtilde{B}}^{\wtilde{\SL}_2}(\mu^{-1}\chi_{\psi_{-1}})$, with respect to the $\wtilde{\SL}_2(k)$-invariant pairing
\begin{equation}\label{sl2pairing}
<\varphi_1,\varphi_2>=\int_k \varphi_1\varphi_2\big(w\smalltwomatrix{1}{x}{}{1}\big)dx,\quad \varphi_1\in \Ind_{\wtilde{B}}^{\wtilde{\SL}_2}(\mu\chi_\psi), \varphi_2\in \Ind_{\wtilde{B}}^{\wtilde{\SL}_2}(\mu^{-1}\chi_{\psi_{-1}}).
\end{equation}
If $k$ is nonarchimedean and not dyadic, $\psi$ is of conductor $\mo_k$, and $\sigma=\Ind_{\wtilde{B}}^{\wtilde{\SL}_2}(\mu\chi_\psi)$ is spherical, then for spherical vectors $\varphi_0\in \sigma$ and $\varphi_0^\vee\in\sigma^\vee$, there is
\begin{equation}\label{sl2mc}
\frac{\big<\sigma(\underline{a})\varphi_0,\varphi_0^\vee\big>}{<\varphi_0,\varphi_0^\vee)}=\frac{|a|\chi_\psi(a)}{1+|\varpi|}\Big(\mu(a)\cdot \frac{1-\mu^{-2}(\varpi)|\varpi|}{1-\mu^{-2}(\varpi)}+\mu^{-1}(a)\cdot \frac{1-\mu^2(\varpi)|\varpi|}{1-\mu^2(\varpi)}\Big),\quad |a|\leqs 1.
\end{equation}

(ii) When $e(\mu)>0$, the intertwining operator $M_{\mu}:\Ind(\mu\chi_\psi)\rar \Ind(\mu^{-1}\chi_\psi)$ is defined by
\begin{equation}\label{sl2inter}
M_{\mu}f(g)=\int_k f\big(\big[w\smalltwomatrix{1}{x}{}{1},1\big]\cdot g\big)dx,\quad f\in \Ind(\mu\chi_\psi).
\end{equation}
For general $\mu$, $M_{\mu}$ is obtained by meromorphically continuing $M_{\mu|\cdot|^{s}}$ from the region with $\re(s)$ large.

Second, type (c) representations in $\tildesr_u$ can be described in more details.
\begin{itemize}
\item When $k$ is nonarchimedean, $\mathrm{St}_\psi(\chi)$ is the unique subrepresentation of $\Ind_{\wtilde{B}}^{\wtilde{\SL}_2}(\chi|\cdot|^{\frac{1}{2}}\chi_\psi)$ and the unique quotient of $\Ind_{\wtilde{B}}^{\wtilde{\SL}_2}(\chi^{-1}|\cdot|^{-\frac{1}{2}}\chi_\psi)$. It fits into two short exact sequences,
\begin{equation}
1\rar \mathrm{St}_\psi(\chi)\rar \Ind_{\wtilde{B}}^{\wtilde{\SL}_2}(\chi|\cdot|^{\frac{1}{2}}\chi_\psi)\rar \omega_{\psi,\chi}^+\rar 1,
\end{equation}
\begin{equation}\label{ses}
1\rar \omega_{\psi,\chi}^+\rar \Ind_{\wtilde{B}}^{\wtilde{\SL}_2}(\chi^{-1}|\cdot|^{-\frac{1}{2}}\chi_\psi)\rar \mathrm{St}_\psi(\chi)\rar 1
\end{equation}

\item Suppose $k=\br$ and $\psi(x)=e^{2\pi i c x}$ with $c>0$. Then $\wtilde{D}_{\ell+\frac{1}{2}}^+$ is holomorphic of lowest weight $\ell+\frac{1}{2}$ and is the unique subrepresentation of $\Ind(|\cdot|^{\ell-\frac{1}{2}}\sgn^\ell\chi_\psi)$. $\wtilde{D}_{\ell+\frac{1}{2}}^-$ is anti-holomoprhic of highest weight $-(\ell+\frac{1}{2})$ and is the unique subrepresentation of $\Ind(|\cdot|^{\ell-\frac{1}{2}}\sgn^{\ell+1}\chi_\psi)$.

\item Each type (c) representation in $\tildesr_u$  is the kernel of the intertwining operator on the induced representation that contains it as a subrepresentation.
\end{itemize}
We also note: If $k$ is nonarchimedean, $\omega_{\psi,\chi}^-$ is supercuspidal; if $k=\br$, $\omega_{\psi,1}^-$ is equivalent to $\wtilde{D}_{3/2}^+$ and $\omega_{\psi,\sgn}^-$ is equivalent to $\wtilde{D}_{3/2}^-$.

\subsubsection{The Whittaker functional}

For a smooth admissible representation $\sigma$ of $\wtilde{\SL}_2(k)$, a Whittaker functional with respect to $\psi$ is an element in $\Hom_{N(k)}(\sigma,\psi)$, that is, a continuous linear functional $\ell_\psi:\sigma\rar \bc$ satisfying
\[
\ell_\psi\big(\sigma\smalltwomatrix{1}{x}{}{1}\varphi\big)=\psi(x)\ell_\psi(\varphi),\quad \varphi\in \sigma,\,\, x\in k.
\]
If $\ell_\psi$ is a Whittaker functional with respect to $\psi$, then $\varphi\rar \ell_\psi(\sigma(\underline{a})\varphi)$ is a Whittaker functional with respect to $\psi_{a^2}$. It is known that $\dim \Hom_{N(k)}(\sigma,\psi)\leqs 1$ if $\sigma$ is irreducible. Suppose $\sigma\in \tildesr$ and $\ell_\psi$ is a nonzero Whittaker functional, the associated Whittaker function associated to $\varphi\in \sigma$ is $W_{\psi,\varphi}(g)=\ell_\psi(\sigma(g)\varphi)$, $g\in \wtilde{\SL}_2(k)$. (Note: when $k$ is archimedean, $\ell_\psi$ is a functional on $\sigma^\infty$.)

Specifically, for $\Ind_{\wtilde{B}}^{\wtilde{\SL}_2}(\mu\chi_\psi)$, the following integral is convergent when $e(\mu)>0$, has holomorphic continuation to all $\mu$, and defines a nonzero Whittaker functional,
\[
\ell_\psi(\varphi)=\int_k \varphi\big(w\smalltwomatrix{1}{x}{}{1}\big)\psi(-x)dx,\quad \varphi\in \Ind_{\wtilde{B}}^{\wtilde{\SL}_2}(\mu\chi_\psi).
\]
When $k$ is nonarchimedean, the holomorphic continuation is simple and there is
\begin{equation}\label{sl2wfpadic}
\ell_\psi(\varphi)=\lim_{\ell\rar \infty}\int_{\varpi^{-\ell}\mo_k} \varphi\big(w\smalltwomatrix{1}{x}{}{1}\big)\psi(-x)dx.
\end{equation}
If $k$ is $p$-adic but not dyadic, $\psi$ is of conductor $\mo_k$, and $\Ind_{\wtilde{B}}^{\wtilde{\SL}_2}(\mu\chi_\psi)$ is spherical, then simple computation shows that for the spherical vector $\varphi_0\in \Ind_{\wtilde{B}}^{\wtilde{\SL}_2}(\mu\chi_\psi)$ there is
\begin{equation}\label{sl2wf}
\frac{\ell_\psi(\underline{a}\circ \varphi_0)}{\ell_\psi(v_0)}=\chi_\psi(a)|a|\cdot\Big(\frac{1-\mu(\varpi)|\varpi|^{\frac{1}{2}}}{1-\mu(\varpi)^2}1_{\mo_k}(a)\mu^{-1}(a)+\frac{1-\mu(\varpi)^{-1}|\varpi|^{\frac{1}{2}}}{1-\mu(\varpi)^{-2}}\cdot 1_{\mo_k}(a)\mu(a)\Big).
\end{equation}

\subsubsection{The estimate of matrix coefficients}

\textbf{From now on until the end of section ~\ref{sl2local}, suppose that $\sigma\in \tildesr_u$ is infinite dimensional and is not an even Weil representation.} We define a number $\alpha(\sigma)\in \br$ by
\begin{equation*}
\alpha(\sigma)=\begin{cases}
0, &\text{if $\sigma$ is tempered},\\
|\alpha|, &\text{if $\sigma=\Ind_{\wtilde{B}}^{\wtilde{\SL}_2}(\mu|\cdot|^{\alpha}\chi_\psi)$.}
\end{cases}
\end{equation*}
Then $\alpha(\sigma)<\frac{1}{2}$. Lemma ~\ref{sl2mbound} below provides a well-known asymptotic estimate for matrix coefficients of $\sigma$, which can be sharpened if $\sigma$ is square-integrable.

Generally, for the matrix coefficients of an irreducible admissible smooth representation of a local group $G$, a complete asymptotic expansion is worked out in \cite{casselman_milicic1982, wallach88} when $G$ is a real reductive group but, to the author's knowledge, is not made explicit when $G$ is a $p$-adic reductive group or the finite cover of such a group. Nevertheless, a principle for deriving the asymptotic expansion in the latter case has been formulated by Casselman \cite[Part I]{casselman2011} in terms of Jacquet modules. For the sake of completeness, we include a proof for Lemma ~\ref{sl2mbound} below.

\begin{lemma}\label{sl2mbound}
Suppose $\varphi_1, \varphi_2\in \sigma$. For any $\epsilon>0$, there is a constant $C_{\epsilon,\varphi_1,\varphi_2}$ such that
\begin{align}
\big|\big(\sigma(\underline{a})\varphi_1,\varphi_2\big)\big|&\leqs C_{\epsilon,\varphi_1,\varphi_2} |a|^{1-\alpha(\sigma)-\epsilon},\quad\quad |a|\leqs 1, \label{sabound}\\
\big|\big(\sigma (w\smalltwomatrix{1}{x}{}{1})\varphi_1,\varphi_2\big)\big|&\leqs C_{\epsilon, \varphi_1,\varphi_2}|x|^{-1+\alpha(\sigma)+\epsilon},\quad |x|\geqs 1. \label{subound}
\end{align}
\end{lemma}
\begin{proof}
Note that (\ref{sabound}) implies (\ref{subound}) by the KAK decomposition of $\wtilde{\SL}_2(k)$ and the $\wtilde{K}$-finiteness of $\varphi_1,\varphi_2$. We now verify (\ref{sabound}). When $k$ is archimedean, (\ref{sabound}) follows from the general estimate in \cite[section 4.3.5]{wallach88}; specifically, if $\sigma$ is of type (a), (\ref{sabound}) can also be verified directly as below.

If $\sigma$ is of type (d), (\ref{sabound}) holds because the matrix coefficient of a supercuspidal representation is compactly supported modulo the center.

Now suppose $\sigma\in \tildesr_u$ is of type (a). There is $\sigma=\Ind_{\wtilde{B}}^{\wtilde{\SL}_2}(\mu\chi_\psi)$, where $\mu$ is unitary or $\mu=\chi|\cdot|^{\alpha}$ with $\chi$ quadratic and $\alpha\in (-\frac{1}{2},0)\cup (0,\frac{1}{2})$. The basic observation is that the matrix coefficient of $\sigma$ can be written in the form $<\sigma(g)\varphi_1,\varphi_2>$, where $\varphi_1\in \sigma$, $\varphi_2\in \Ind_{\wtilde{B}}^{\wtilde{\SL}_2}(\mu^{-1}\chi_{\psi_{-1}})$, and the pairing $<,>$ is as in (\ref{sl2pairing}). So
\begin{align}
<\sigma(\underline{a})\varphi_1,\varphi_2>=&\int_k \varphi_1\big(\big[w\smalltwomatrix{1}{x}{}{1},1\big]\cdot \big[\smalltwomatrix{a}{}{}{a^{-1}},1\big]\big)\varphi_2\big(\big[w\smalltwomatrix{1}{x}{}{1},1\big]\big)dx \label{equation1}\\
=&\chi_\psi(a)\mu^{-1}(a)|a|\int_k \varphi_1\big(\big[w\smalltwomatrix{1}{x}{}{1},1\big]\big)\varphi_2\big(\big[w\smalltwomatrix{1}{a^2x}{}{1},1\big]\big)dx. \label{equation2}
\end{align}

Define smooth functions $\Phi_i(x):=\varphi_i\big(\big[w\smalltwomatrix{1}{x}{}{1},1\big]\big)$ and $\phi_i(x)=\varphi_i\big(\big[\smalltwomatrix{1}{}{x}{1},1\big]\big)$ on $k$. Then the relation
\[
\big[w\smalltwomatrix{1}{x}{}{1},1\big]=\big[\smalltwomatrix{1}{-x^{-1}}{}{1},1\big]\big[\smalltwomatrix{-x^{-1}}{}{}{-x},1\big]\big[\smalltwomatrix{1}{}{x^{-1}}{1}\big],
\]
leads to
\begin{equation}\label{frelation}
\Phi_1(x)=\chi_\psi(-x)\mu^{-1}(-x)|x|^{-1}\phi_1(\frac{1}{x})\quad, \Phi_2(x)=\chi_{\psi_{-1}}(-x)\mu(-x)|x|^{-1}\phi_2(\frac{1}{x}).
\end{equation}

When $|a|\leqs 1$, it is good to estimate the integral $\int_k \Phi_1(x)\Phi_2(a^2x)dx$ in (\ref{equation2}) by dividing it into three parts: (I) $|x|\leqs 1$, (II) $|x|\geqs |a|^{-2}$, and (III) $1<|x|<|a|^{-2}$. Observe that by (\ref{frelation}),
\begin{align*}
\int_{|x|\geqs |a|^{-2}} \Phi_1(x)\Phi_2(a^2x)dx=&\chi_\psi(a)\chi_a(x)\mu^2(a)|a|^{-2}\int_{|x|\geqs |a|^{-2}} |x|^{-2}\phi_1\big(\frac{1}{x}\big)\phi_2\big(\frac{1}{a^2x}\big)dx\\
\int_{1<|x|< |a|^{-2}} \Phi_1(x)\Phi_2(a^2x)dx=&\int_{1<|x|< |a|^{-2}} \chi_\psi(-x)\mu^{-1}(-x)|x|^{-1}\phi_1\big(\frac{1}{x}\big)\Phi_2(a^2x)dx.
\end{align*}
Hence the integral of $\Phi_1(x)\Phi_2(a^2x)$ on region (I) is bounded in mganitude by a constant $C_1$, on region (II) by $C_2|\mu(a)|^2$ for certain constant $C_2$, on region (III) by $C_{3}(1+|a|^{2e(\mu)})$ if $e(\mu)>0$ or $C_3(1+\big|\ln |a|\big|)$ if $e(\mu)\neq 0$ for certain constant $C_3$. Therefore,
\[
\big|<\sigma(\underline{a})\varphi_1,\varphi_2>\big|=|a|^{-e(\mu)+1}\cdot \big|\int_k \Phi_1(x)\Phi_2(a^2x)dx\big|\leqs C_\epsilon |a|^{1-\epsilon} (|a|^{e(\mu)}+|a|^{-e(\mu)}).
\]
This proves (\ref{sabound}) when $\sigma\in \tildesr_u$ is of type (a), for both archimedean and nonarchimedean $k$.

$\ $\\
\indent The remaing case is the Steinberg representation $\sigma=\mathrm{St}_\psi(\chi)$ with $\chi$ quadratic. Regard it as a subrepresentation of $\sigma=\Ind_{\wtilde{B}}^{\wtilde{\SL}_2}(\mu\chi_\psi)$ with $\mu=\chi|\cdot|^{1/2}$. As in the former situation, a matrix coefficient of $\sigma$ is of the form $<\sigma(g)\varphi_1,\varphi_2>$, where $\varphi_1\in \sigma$ and $\varphi_2\in \Ind_{\wtilde{B}}^{\wtilde{\SL}_2}(\mu^{-1}\chi_{\psi_{-1}})$. The extra property we now have is
\begin{equation}\label{annihilation}
\int_k \varphi_1\big(\big[w\smalltwomatrix{1}{x}{}{1},1\big]g\big)dx=0,\quad  g\in \wtilde{\SL}_2(k),
\end{equation}
because vectors in $\mathrm{St}_\psi(\chi)$ are annihilated by the intertwining operator. On the other hand, because $k$ is nonarchimedean, there exists $\delta$ such that $\varphi_2\big(\big[w\smalltwomatrix{1}{x}{}{1},1\big]\big)=\varphi_2([w,1])$ when $|x|\leqs \delta$. Keeping the notations $\Phi_i(x), \phi_i(x)$ as before, one could use (\ref{annihilation}) to rewrite (\ref{equation1}),
\begin{align*}
<\sigma(\underline{a})\varphi_1,\varphi_2>=&\int_{|x|\geqs \delta} \varphi_1\big(\big[w\smalltwomatrix{1}{x}{}{1},1\big]\cdot \big[\smalltwomatrix{a}{}{}{a^{-1}},1\big]\big)\Big(-\varphi_2([w,1])+\varphi_2\big(\big[w\smalltwomatrix{1}{x}{}{1},1\big]\big)\Big)dx \\
=&\chi_\psi(a)\mu^{-1}(a)|a|\int_{|x|\geqs \delta |a|^{-2}} \varphi_1\big(\big[w\smalltwomatrix{1}{x}{}{1},1\big]\big)\Big(-\varphi_2([w,1])+\varphi_2\big(\big[w\smalltwomatrix{1}{a^2x}{}{1},1\big]\big)\Big)dx.
\end{align*}
Then one can apply (\ref{frelation}) to the right hand side and do a simple estimate to get $|<\sigma(\underline{a})\varphi_1,\varphi_2>|\leqs C|a|^{3/2}$. (Recall $\mu=\chi|\cdot|^{\frac{1}{2}}$.) This is a better estimate than (\ref{abound}) since $\alpha(\sigma)=0$ in this case.

\end{proof}

\subsubsection{Express Whittaker functionals with the inner product}

Let $T_{(\sigma(n)\varphi_1,\varphi_2)}$ be the tempered distribution on $k$ associated to the bounded smooth function $(\sigma(n)\varphi_1,\varphi_2)$. By the same argument as in Remark \ref{gl2regular},  $\mf T_{(\sigma(n)\varphi_1,\varphi_2)}$ is represented over $k^\times$ by a smooth function $W_{\varphi_1,\varphi_2,\psi}(t)$.
\begin{definition}
We define
\[
\int_k \big(\sigma\smalltwomatrix{1}{x}{}{1}\varphi_1,\varphi_2)\psi(-\delta n)dn\overset{\triangle}{=}W_{\varphi_1,\varphi_2,\psi}(-\delta),\quad \delta\in k^\times.
\]
\end{definition}
\noindent Obviously, $W_{\cdot,\cdot,\psi}(-\delta)$ is in $\Hom_{N(k)}(\sigma,\psi_\delta)\otimes \overline{\Hom_{N(k)}(\sigma,\psi_\delta)}$.
\begin{remark}
By defintion, for any sequence $\{\phi_m\}$ in $\ms(k)$ satisfying $T_{\phi_m}\rar \delta_t$ in $\ms^\prime(k)$, there is
\begin{equation}\label{sl2abelintegral}
W_{\varphi_1,\varphi_2,\psi}(t)=\lim_{m\rar \infty} \int_k (\sigma(n)\varphi_1,\varphi_2)\hat{\phi}_m(n)dn.
\end{equation}
As a consequence, if $W_{\cdot,\cdot,\psi}(-\delta)=c\ell_{\psi_\delta}\overline{\ell_{\psi_\delta}}$ for certain $\ell_{\psi_\delta}\in \Hom_{N(k)}(\sigma,\psi_\delta)$, then
\begin{equation}\label{generalb}
W_{\varphi_1,\varphi_2,\psi}(-\delta a^2)=c|a|^{-2}\ell_\psi(\sigma(\underline{a})\varphi_1)\overline{\ell_\psi(\sigma(\underline{a})\varphi_2)},\quad \forall\, \varphi_1,\varphi_2\in \sigma,\, a\in k^\times.
\end{equation}
\end{remark}

$\ $\\
\indent Now we choose a set of representative $\{\delta_i\}$ of $k^\times/{k^\times}^2$, with $\delta_1=1$. For each $\delta_i$, choose a nonzero Whittaker functional $\ell_{\psi_{\delta_i}}$ if $\Hom_{N(k)}(\sigma,\psi_{\delta_i})\neq 0$ and set $\ell_{\psi_{\delta_i}}=0$ otherwise. Because $\dim_{N(k)}(\sigma,\psi_{\delta_i})\leqs 1$, there are constants $c_{\sigma,\delta_i}$ such that
\begin{equation}\label{m2wsl22}
\int_k (\sigma(n)\varphi_1,\varphi_2)\psi(-\delta_i n)dn=W_{\varphi_1,\varphi_2,\psi}(-\delta_i)=\frac{c_{\sigma,\delta_i}}{|\delta_i|}\ell_{\psi_{\delta_i}}(\varphi_1)\overline{\ell_{\psi_{\delta_i}}(\varphi_2)}. 
\end{equation}
We take $c_{\sigma,\delta_i}=0$ if $\ell_{\psi_{\delta_i}}=0$. The lemma below shows that $c_{\sigma,\delta_i}$ must be nonzero if $\ell_{\psi_{\delta_i}}\neq 0$.

\begin{lemma}\label{m2wsl2}
\emph{(i)} There exists $\delta\in k^\times$ such that $W_{\cdot,\cdot,\psi}(-\delta)$ is a nonzero functional on $\sigma\otimes \sigma$.\\
\emph{(ii)} For $\delta\in k^\times$, $W_{\cdot,\cdot,\psi}(-\delta)$ is a nonzero functional on $\sigma\otimes \sigma$ if and only if $\Hom_{N(k)}(\sigma,\psi_\delta))\neq 0$.
\end{lemma}
\begin{proof}
(i) If the assertion is not true, then for any $\varphi_1,\varphi_2\in \sigma$, the function $W_{\varphi_1,\varphi_2,\psi}(a)$ is zero on $k^\times$. So $\mf T_{(\sigma(n)\varphi_1,\varphi_2)}$ is a tempered distribution supported on the single point set $\{0\}$. As argued in the proof of Lemma ~\ref{gl2localduality}, $(\sigma(n)\varphi_1,\varphi_2)$ would be a constant if $k$ is nonarchimedean or a polynomial if $k$ is archimedean. Because $(\sigma(n)\varphi_1,\varphi_2)$ decays to zero at infinity by (\ref{subound}), it should be zero. But this is a contradiction when $\varphi_1=\varphi_2\neq 0$.

(ii) The `` only if " part is obvious. For the `` if " part, we show there would be a contradiction if $\Hom_{N(k)}(\sigma,\psi_{\delta})$ is nonzero but $W_{\cdot,\cdot,\psi}(-\delta)$ is zero. Denote by $\cl(\sigma)$ the unitary closure of $\sigma$, which is the Hilbert space completion of $\sigma$. If $\Hom_{N(k)}(\sigma,\psi_{\delta})$ is nonzero, choose a nonzero element $\ell_{\psi_\delta}$ therein and let $v_{\psi_\delta}\in \cl(\sigma)$ be such that $\ell_{\psi_\delta}(\varphi)=(\varphi, v_{\psi_\delta})$ for all $\varphi\in \sigma$. Then $\sigma(n)v_{\psi_\delta}=\psi(\delta n)v_{\psi_\delta}$. On the other hand, if $W_{\cdot,\cdot,\psi}(-\delta)$ is zero, then by (\ref{generalb}) $W_{\varphi_1,\varphi_2,\psi}(-\delta a^2)=0$ for all $\varphi_1, \varphi_2\in \sigma$, $a\in k^\times$.

Let $\{\varphi_m\}$ be a sequence of vectors in $\sigma$ converging to $v_{\psi_\delta}$ in $\cl(\sigma)$. Let $\lambda(t)\in C^\infty_c(k^\times)$ be supported on $-\delta\cdot {k^\times}^2$ with $\lambda(-\delta)=1$, then
\[
0=\int_k W_{\varphi_m,\varphi_m,\psi}(t)\lambda(t)=\int_k (\sigma(n)\varphi_m,\varphi_m)\widehat{\lambda}(n)dn.
\]
Because $\varphi_m\rar v_{\psi_\delta}$ implies $(\sigma(n)\varphi_m,\varphi_m)$ uniformly converges to $(\sigma(n)v_{\psi_\delta}, v_{\psi_\delta})$, there is
\[
0=\int_k (\sigma(n)v_{\psi_\delta},v_{\psi_\delta})\widehat{\lambda}(n)dn=(v_{\psi_\delta},v_{\psi_\delta})\int_k \psi(\delta n)\widehat{\lambda}(n)dn=(v_{\psi_\delta},v_{\psi_\delta})\lambda(-\delta)\neq 0.
\]
This is obviously a contradiction.
\end{proof}

\subsubsection{The estimate of Whittaker functions}

The asymptotic estimate of matrix coefficients naturally leads to an according estimate of Whittaker functions. We demonstrate this mechanism when $k$ is nonarchimedean in the proof of Lemma ~\ref{sl2wbound}.

Generally, for Whittaker functions on a local group $G$, a complete asymptotic expansion is worked out in \cite{wallach92} when $G$ is a real reductive group and a principle for deriving the asymptotic expansion in terms of Jacquet modules is given in \cite[Section 6]{casselman_shalika80} when $G$ is a $p$-adic reductive group. \cite{lapid_mao2009} further works out an explicit asymptotic formula in the latter case. However, to the author's knowledge, a similar asymptotic expansion has not been made explicit on $p$-adic covering groups. For the sake of completeness, we include a brief proof for Lemma ~\ref{sl2wbound} below.

\begin{lemma}\label{sl2wbound}
Suppose $\Hom_{N(k)}(\sigma,\psi)\neq 0$ and $ \varphi\in \sigma$. When $a$ is near infinity, $W_{\varphi,\psi}(\underline{a})$ vanishes if $k$ is nonarchimdean and rapidly decays if $k$ is archimedean. For any $\epsilon>0$, there is a constant $C_{\epsilon,\varphi}$ such that
\begin{equation}
\big|W_{\varphi,\psi}(\underline{a})\big|\leqs C_{\epsilon,\varphi}|a|^{1-\alpha(\sigma)-\epsilon},\quad\quad\quad\quad\quad |a|\leqs 1, \label{swbound}
\end{equation}
\end{lemma}
\begin{proof}
When $k$ is archimdean, the assertion follows from the general estimate in \cite[section 15.2.2]{wallach92}. If $k$ is complex, the assertion also follows from the according statement on $\GL_2(\bc)$-representations. For nonarchimedean $k$, we derive the assertion below from the estimate of matrix coefficients.

By (\ref{generalb}), $|W_{\varphi,\psi}(\underline{a})|^2$ is proportional to $|a|^2\cdot W_{\varphi,\varphi,\psi}(-a^2)$. So one may work on the latter. Set $\phi_m(t)=|\varpi|^{-m} 1_{\varpi^m\mo_k}(t+a^2)$, then $T_{\phi_m}\rar \delta_{-a^2}$ in $\ms^\prime(k)$. By (\ref{sl2abelintegral}), we have the following working formula for $W_{\varphi,\varphi,\psi}(-a^2)$,
\begin{equation}\label{wformulap}
W_{\varphi,\varphi,\psi}(-a^2)=\lim_{m\rar \infty} \int_{\varpi^{-m}\mo_k} (\sigma(n)\varphi,\varphi)\psi(-a^2n)dn
\end{equation}
Because $\sigma$ is smooth, there is $m_0\in \bn$ such that $\sigma(n)\varphi=\varphi$ when $n\in \varpi^{m_0}\mo_k$ and $\sigma(\underline{b})\varphi=\varphi$ when $b\in 1+\varpi^{m_0}\mo_k$. We may suppose $m_0$ is not too small, so that $U^{(m_0)}:=1+\varpi^{m_0}\mo_k$ is a group and contained in ${k^\times}^2$. 

When $|a|$ is large, $\psi(-a^2n)$ is nontrivial on $\varpi^{m_0}\mo_k$. As $(\sigma(n)\varphi,\varphi)$ is constant on $\varpi^{m_0}\mo_k$, the integral $\int_{\varpi^{-m}\mo_k}$ in (\ref{wformulap}) must vanish when $m\geqs -m_0$ and hence the limit therein is zero. This shows that $W_{\varphi,\varphi,\psi}(-a^2)$  and $W_{\varphi,\psi}(\underline{a})$ vanish near infinity.

Now suppose $|a|\leqs 1$. Write $a=\varpi^{m_1}$, with $m_1\geqs 0$. Write $\mathrm{Cond}(\psi)=\varpi^{m_2}\mo_k$. Then $\psi(-a^2n)$ is nontrivial on $\varpi^{m_2-2m_1-1}\mo_k$. On the other hand, because of the relation
\[
[\smalltwomatrix{1}{b^2x}{}{1},1]=\underline{b}\cdot [\smalltwomatrix{1}{x}{}{1},1]\cdot \underline{b}^{-1},
\]
the function $(\sigma(n)\varphi,\varphi)$ is constant on any $U^{(m_0)}$-coset. It follows that when $m\geqs m_0+2m_1+1-m_2$,
\[
\int_{\varpi^{-m}\mo_k^\times}(\sigma(n)\varphi,\varphi)\psi(-a^2n)dn=\sum_{b\in \mo_k^\times/U^{(m_0)}} \big(\sigma(\varpi^{-m}b)\varphi,\varphi\big)\int_{\varpi^{-m}bU^{(m_0)}}\psi(-a^2n)dn=0.
\]
So the integral in (\ref{wformulap}) stabilizes from $m=m_0+2m_1-m_2$ onwards, whence
\[
W_{\varphi,\varphi,\psi}(-a^2)=\int_{\varpi^{-(m_0+2m_1-m_2)}\mo_k} (\sigma(n)\varphi,\varphi)\psi(-a^2n)dn
\]

Now the bound of matrix coefficients (\ref{subound}) implies
\[
|W_{\varphi,\varphi,\psi}(-a^2)|\leqs \int_{|n|\leqs |a|^{-2}|\varpi|^{m_2-m_0}} C_{\epsilon,\varphi,\varphi}|n|^{-1+\alpha(\sigma)+\epsilon}dn\leqs C^\prime_{\epsilon,\varphi}|a|^{-2\alpha(\sigma)-2\epsilon},
\]
where $C^\prime_{\epsilon,\varphi}=C_{\epsilon,\varphi,\varphi}|\varpi|^{(m_2-m_0)\cdot(\alpha(\sigma)+\epsilon)}$. This leads to (\ref{swbound}).
\end{proof}

\subsubsection{Express the inner product with Whittaker functionals}

The following is the dual of (\ref{m2wsl22}).

\begin{lemma}\label{lemmaw2m}
Recall the constants $c_{\sigma,\delta_i}$ in (\ref{m2wsl22}) and that $\delta_1=1$. There is
\begin{equation}
(\varphi_1,\varphi_2)=\frac{|2|_k}{2}\sum_i c_{\sigma,\delta_i}\int_{k^\times} W_{\varphi_1,\psi_{\delta_i}}(\underline{a})\overline{W_{\varphi_2,\psi_{\delta_i}}(\underline{a})}d^\times a. \label{w2msl2}
\end{equation}
\end{lemma}
\begin{proof}
The integrals on the right hand side are convergent due to the estimate in Lemma ~\ref{sl2wbound}. By (\ref{m2wsl22}) and (\ref{generalb}), the right hand side is actually equal to
\[
\sum_i \int_k 2^{-1}|2\delta_i a^2|_k W_{\varphi_1,\varphi_2,\psi}(-\delta_i a^2) d^\times a=\int_k W_{\varphi_1,\varphi_2,\psi}(t)dt
\]
Recall that $W_{\varphi_1,\varphi_2,\psi}(t)$ represents $\mf T_{(\sigma(n)\varphi_1,\varphi_2)}$ over $k^\times$. With the same argument for (\ref{w2mgl2}), one can show that the inverse Fourier transform of $W_{\varphi_1,\varphi_2,\psi}(t)$, which is $\int_k W_{\varphi_1,\varphi_2,\psi}(t)\psi(-tn)dt$, equals with $(\sigma(n)\varphi_1,\varphi_2)$. Then setting $n=0$ yields the desired equality.
\end{proof}

\begin{remark}
(\ref{w2msl2}) was previously deduced in \cite{baruch_mao03, baruch_mao05} when $k$ is nonarchimedean or real, based on case-by-case study of Kirillov models on the two-fold cover of $\GL_2(k)$. Our view is that it is one case of Fourier inversion formula, whose applicability merely relies on simple bounds of Whittaker functions and matrix coefficients. (Note that the ``unitary" bounds are more than enough here.)
\end{remark}

\subsubsection{An integration formula}

We prove a local equality that shows how to integrate the matrix coefficient of $\sigma$ along the unipotent subgroup $N$.
\begin{lemma}\label{identity-mbintegral}
Suppose $\Phi\in \ms(k)$, then
\begin{equation}\label{iformula}
\int_k \overline{(\sigma(n)\varphi_1,\varphi_2)}\mf_2\Phi(n)dn=c_{\sigma,\psi}\int_{k^\times} \overline{W_{\varphi_1,\psi}}W_{\varphi_2,\psi}(\underline{a})\Phi(a)|a|^{-1}d^\times a.
\end{equation}
\end{lemma}
\begin{proof}
Set $W_j(a)=W_{\varphi_1,\psi_{\delta_j}}\overline{W_{\varphi_2,\psi_{\delta_j}}}(\underline{a})|a|^{-1}$. Then $(\sigma(n)\varphi_1,\varphi_2)=\sum_j \frac{c_{\sigma,\psi_{\delta_j}}|2|_k}{2}\mf_2 W_j(\delta_j n)$ by Lemma ~\ref{lemmaw2m}. It follows that
\[
\int_k \overline{(\sigma(n)\varphi_1,\varphi_2)}\mf_2\Phi(n)dn=\int_k \Big[\sum_{\delta_j} \frac{c_{\sigma,\psi_{\delta_j}}|2|_k}{2}\cdot\mf_2 \overline{W_j}(-\delta_j n)\Big]\cdot\mf_2\Phi(n)dn.
\]
Because $(\sigma(n)\varphi_1,\varphi_2)|n|^{-\frac{1}{2}}$, $W^2_j(a)|a|^{-1}$, and $W_j(a)|a|^{-1}$ are in $L^1(k)$ by Lemma ~\ref{sl2mbound} and Lemma ~\ref{sl2wbound}, we can apply Proposition ~\ref{isometry-identity-extra} to the right hand side of the above equation. This leads to (\ref{iformula}).
\end{proof}

\subsection{Global Theory}\label{sl2t_global}

Let $\wtilde{\SL}_2(\ba)$ denote the two-fold metaplectic cover of $\SL_2(\ba)$. Choose a maximal compact subgroup $K_v$ of $\SL_2(F_v)$ at every local place of $F$ and write $K=\prod_v K_v$.  Let $\wtilde{K}$ be the preimage of $K$ in $\wtilde{\SL}_2(\ba)$. There are two types of genuine cuspidal automorphic forms on $\wtilde{\SL}_2(\ba)$, elementary theta series and their orthogonal complement in the cuspidal spectrum.

Fix a non-trivial character $\psi$ of $\ba/F$. For $a\in F^\times$, let $\omega_{\psi_a}=\otimes_v \omega_{\psi_{a,v}}$ be the global Weil representation of $\wtilde{\SL}_2(\ba)$ on $\ms(\ba)$, where $\omega_{\psi_{a,v}}$ is as in section ~\ref{localweilrep}. For each $\wtilde{K}$-finite function $\phi\in \ms(\ba)$, the associated theta function $\theta_{\phi,\psi_a}(g)$ is called an elementary theta series,
\[
\Theta_{\phi,\psi_\delta}(g)=\sum_{\xi\in F}\omega_{\psi_\delta}(g)\phi(\xi),\quad g\in \wtilde{\SL}_2(\ba).
\]
$\Theta_{\phi,\psi_a}(g)$ is a genuine automorphic form on $\wtilde{\SL}_2(\ba)$; it is cuspidal if and only if $\phi(0)=0$.

Let $\ma_0(\wtilde{\SL}_2)$ denote the space of genuine cuspidal automorphic forms on $\wtilde{\SL}_2(\ba)$ and $\ma_{00}(\wtilde{\SL}_2)$ the subspace consisting of forms that are orthogonal to all elementary theta series. According to \cite{wald80, wald91}, $\ma_{00}(\wtilde{\SL}_2)$ is a disjoint union of Waldspurger packets $Wd_\psi(\pi)$,
\[
\ma_{00}(\wtilde{\SL}_2)=\sqcup_\pi Wd_\psi(\pi),
\]
where $\pi$ runs over irreducible cuspidal automorphic $\PGL_2(\ba)$-representations and $Wd_\psi(\pi)$ consists global theta lifts $\sigma=\Theta_{\wtilde{\SL}_2\times \PGL_2}(\pi\otimes \chi_\delta,\psi_\delta)$ that are nonvanishing, with $\delta\in F^\times/{F^\times}^2$. The dependence of the packet on the choice of $\psi$ is $Wd_{\psi}(\pi)=Wd_{\psi_\delta}(\pi\otimes \chi_\delta)$.

\subsubsection{The global theta lift between $\wtilde{\SL}_2$ and $\PGL_2$}\label{weilrepso3}

Consider the quadratic space $(V,q)$, where
\[
V=\{X\in M_{2\times 2}(F):\Tr(X)=0\},\quad q(X)=-\det X.
\]
$\GL_2$ acts on $V$ by $h\circ X=hXh^{-1}$ and this leads to an isomorphism $\PGL_2\isoto \SO(V)$. 

Let $\omega_{\psi,V}=\otimes_v \omega_{\psi_v, V_{F_v}}$ be the global Weil representation of $\PGL_2(\ba)\times \OO(V)_\ba$ on $\ms(V_\ba)$ with respect to $\psi$. When $v$ is an archimedean place, let $\mathrm{S}(V_{F_v})$ be the associated Fock model of $\ms(V_{F_v})$---it is the associated $(\fs\fp_6,\wtilde{\mathrm{K}})$-module of $\omega_{\psi_v, V_{F_v}}$, where $\wtilde{\mathrm{K}}$ is a maximal compact subgroup of $\wtilde{\Sp}_6(F_v)$. When $v$ is nonarchimedean, let $\mathrm{S}(V_{F_v})$ be just $\ms(V_{F_v})$. Denote by $\mathrm{S}(V_\ba)$ the resticted tensor product $\otimes_{v} \mathrm{S}(V_{F_v})$. For $\phi\in \mathrm{S}(V_\ba)$, define the kernel function

\[
\Theta_{\phi,\psi}(g,h)=\sum_{\xi\in V(F)}\omega_{\psi,V}(g,h)\phi(\xi),\quad g\in \wtilde{\SL}_2(\ba), h\in \OO(V)_\ba.
\]

Let $\sigma$ be an irreducible genuine cuspidal automorphic representation of $\wtilde{\SL}_2(\ba)$. For a form $\varphi\in \sigma$, define its lift to $\PGL_2$ with respect to $\omega_\psi$ via $\phi\in \mathrm{S}(V_\ba)$ as
\[
\Theta_\psi(\phi,\varphi)(h)=\int_{\SL_2(\bq)\backslash \SL_2(\ba)}\overline{\varphi(g)}\Theta_{\phi,\psi}(g,h)dg,\quad h\in \PGL_2(\ba).
\]
The global theta lift of $\sigma$ to $\PGL_2$ with respect to $\psi$ is then
\[
\Theta(\sigma,\psi)=\{\Theta_\psi(\phi,\varphi)| \phi\in \mathrm{S}(V_\ba), \varphi\in \sigma\}.
\]

Similarly, let $\pi$ be an irreducible cuspidal automorphic representation of $\PGL_2(\ba)$. For $f\in \pi$, define its lift by $\Theta_\psi(\phi,f)(g)=\int_{\PGL_2(F)\backslash \PGL_2(\ba)}\overline{f(h)}\Theta_{\phi,\psi}(g,h)dh$. The global theta lift of $\pi$ to $\wtilde{\SL}_2(\ba)$ with respect to $\psi$ is then $\Theta(\pi,\psi)=\{\Theta_\psi(\phi,f)| \phi\in \mathrm{S}(V_\ba), f\in \pi\}$.

We note that $\Theta(\sigma,\psi)$ (resp. $\Theta(\pi,\psi)$) is either zero or irreducible cuspidal. When $\Theta(\sigma,\psi)$ is nonzero, the local component $\sigma_v\in \tildesr_u$ is infinite dimensional and is not an even Weil representation, according to explicit local theta correspondence between $\PGL_2$ and $\wtilde{\SL}_2$.

\subsubsection{The global Whittaker functional}\label{sl2_gwf}

Let $\sigma$ be an irreducible cuspidal $\wtilde{\SL}_2(\ba)$-representation contained in $\ma_{00}(\wtilde{\SL}_2)$. Equip $\SL_2(\ba)$ and $\ba$ with the Tamagawa measures. The standard inner product and the standard Whittaker functional (with respect to $\psi$) on $\sigma$ are
\[
(\varphi_1,\varphi_2)=\int_{\SL_2(\bq)\backslash \SL_2(\ba)}\varphi_1(g)\overline{\varphi_2(g)}dg,\quad \ell_{\psi}(\varphi)=\int_{\ba/F}\varphi(n)\psi(- n)dn.
\]
According to \cite{wald80}, $\ell_\psi$ is nonzero on $\sigma$ if and only if the global theta lift of $\sigma$ to $\PGL_2$ is nonzero.

Now suppose $\ell_\psi$ is nonzero and write $\pi=\Theta(\sigma,\psi)$. We will explicate the relation between $(\, , \,)$ and $\ell_\psi$ below. First, Write $\sigma=\otimes \sigma_v$ as a restricted tensor product of irreducible $\wtilde{\SL_2}(F_v)$-representations. Because $\sigma_v$ and $\pi_v$ are in local theta correspondence, $\sigma_v$ is infinite dimensional and can not be an even Weil representation. For almost all $v$, let $\varphi_{v,0}$ be the spherical vector chosen for constructing the restricted tensor product. At each place $v$, choose a local inner product $(,)_v$ on $\sigma_v\otimes \sigma_v$ such that $(\, ,\,)=\prod_v (\, ,\,)_v$ and $(\varphi_{v,0},\varphi_{v,0})=1$ for almost all $v$.

Second, pick up a set of representatives $\{\delta_{v,i}\}$ of $F_v^\times/{F_v^\times}^2$ with $\delta_{v,1}=1$. For each $\delta_{v,i}$, choose a non-zero local Whittaker functional $\ell_{\psi_{\delta_{v,i}}}$ on $\sigma_v$ with respect to $\psi_{\delta_{v,i}}$ if $\Hom_{N(F_v)}(\sigma_v,\psi_{\delta_{v,i}})\neq 0$, or set $\ell_{\psi_{\delta_{v,i}}}=0$ otherwise. We require $\ell_\psi=\prod_v \ell_{\psi_v}$ and $\ell_{\psi_v}(\varphi_{v,0})=1$ for almost all spherical $\sigma_v$. Set $W_{\varphi_v,\psi_{\delta_{v,i}}}(g_v)=\ell_{\psi_{\delta_{v,i}}}(\sigma_v(g_v)\varphi_v)$.

Third, at each place $v$, the local theory yields local constants $c_{\sigma_v,\psi_v}$ and $c_{\sigma_v,\psi_{\delta_{i,v}}}$ for each $i\neq 1$ such that for all $\varphi_{1,v}, \varphi_{2,v}\in \sigma_v$, there are
\begin{align}
(\varphi_{1,v},\varphi_{2,v})_v=&\frac{|2|_{F_v}\cdot c_{\sigma_v,\psi_v}}{2}\int_{F_v^\times}W_{\varphi_{1,v},\psi_v}\overline{W_{\varphi_{2,v},\psi_v}}(\underline{a_v})d^\times a_v \label{formula21}\\
\nonumber &+\sum_{i\neq 1}\frac{|2|_{F_v}\cdot c_{\sigma_v,\psi_{v,\delta_i}}}{2}\int_{k^\times} W_{\varphi_1,\psi_{\delta_{v,i}}}\overline{W_{\varphi_2,\psi_{\delta_{v,i}}}}(\underline{a}_v)d^\times a_v,\\
\ell_{\psi_v}(\varphi_{1,v})\overline{\ell_{\psi_v}(\varphi_{2,v})}=&\frac{1}{c_{\sigma_v,\psi_v}}\int_k \big(\sigma_v(n_v)\varphi_{1,v},\varphi_{2,v}\big)\psi(-n_v)dn_v. \label{formula22}
\end{align}
With (\ref{sl2mc}) and the second formula above, one can easily see that $c_{\sigma_v,\psi_v}=\frac{L(\frac{1}{2},\pi_v)\zeta_{F_v}(2)}{L(1,\pi_v,\ad)}$ at almost all places. So we introduce the normalized local pairing below,
\[
\ml_{\psi_v}^\sharp(\varphi_{1,v},\varphi_{2,v})=\frac{L(1,\pi_v,\ad)}{L(\frac{1}{2},\pi_v)\zeta_{F_v}(2)}\int_k \big(\sigma_v(n_v)\varphi_{1,v},\varphi_{2,v}\big)\psi(-n_v)dn_v.
\]
\begin{proposition}\label{sl2wpfa}
Put $c_\sigma=\frac{L(1,\pi,\ad)}{L(\frac{1}{2},\pi)\zeta_F(2)}\prod_v \frac{L(\frac{1}{2},\pi_v)\zeta_{F_v}(2)}{c_{\sigma_v,\psi_v}L(1,\pi_v,\ad)}$. For decomposable $\varphi_1,\varphi_2\in \sigma$ there is
\[
\ell_\psi(\varphi_1)\overline{\ell_\psi(\varphi_2)}=c_\sigma \cdot \frac{L(\frac{1}{2},\pi)\zeta_F(2)}{L(1,\pi,\ad)}\prod_v\ml_{\psi_v}^\sharp(\varphi_{1,v},\varphi_{2,v}).
\]
\end{proposition}
\begin{proof}
The formula is simply the product of (\ref{formula22}) at all local places.
\end{proof}

\section{The Transfer of Whittaker functionals}\label{sect_tran}

In this section, we show that the constant $c_\sigma$ in Proposition ~\ref{sl2wpfa} takes value $\frac{1}{2}$. The idea is to write $\pi=\Theta(\sigma,\psi)$ and use theta lifting to express the Whittaker functional on $\pi$ in terms of the Whittaker functional on $\sigma$. This leads to the relation $c_\sigma=c_\pi$ and then Proposition ~\ref{compare_gl2ip} is applied.

Recall the Weil representation of $\wtilde{\SL}_2(\ba)\times \PGL_2(\ba)$ on $\ms(V_\ba)$ with respect to a fixed additive character $\psi$ of $\ba/F$, where $V$ is as in section ~\ref{weilrepso3} and the formulas for Weil representation are as in section ~\ref{localweilrep}. Choose a basis $\{e_+, e_0, e_-\}$ of $V(F)$, with
\[
e_+=\smalltwomatrix{0}{1}{0}{0},\quad e_0=\smalltwomatrix{1}{}{}{-1},\quad e_-=\smalltwomatrix{0}{0}{1}{0}.
\]
Define the partial Fourier transform from $\ms(V_\ba)$ to $\ms(\ba \oplus \ba^2)$ by
\[
\widehat{\phi}(x_0;x_-,y_-)=\int_{\ba}\phi(x_+ e_+ +x_0e_0 + x_-e_-)\psi\big(x_+ y_-\big)dx_+.
\]

We first transfer the global Whittaker period functional on $\pi$. It was used by Waldspurger to show ``$\ell_\psi$ is nonzero on $\sigma$ $\Rar\Theta_\psi(\sigma,\psi)\neq 0$".
\begin{lemma}\label{transfer_global}
Suppose $f=\Theta_\psi(\phi,\varphi)\in \pi$ with $\phi\in \mathrm{S}(V_\ba)$ and $\varphi\in \pi$, then
\begin{equation}\label{gptransfersl2}
\ell_{\pi,\psi_{-2}}(f)=\int_{N(\ba)\backslash \SL_2(\ba)}\overline{W_{\varphi,\psi}(g)}\omega(g)\widehat{\phi}(1;0,1)dg.
\end{equation}
\end{lemma}
\begin{proof}
One first observes that
\begin{align*}
\Theta_\phi(g)=&\sum_{\xi\in F}\sum_{\eta\in F^2}\omega(g)\widehat{\phi}(\xi;\eta)=\sum_{\xi\in F}\widehat{\phi}(\xi;0)+\sum_{\gamma\in N(F)\backslash \SL_2(F)}\sum_{\xi\in F}\omega(\gamma g)\widehat{\phi}(\xi;0,1).
\end{align*}
Because the first summand is an elementary theta series and $\sigma\subset \ma_{00}(\wtilde{\SL}_2)$, there is
\begin{align*}
\Theta_\psi(\phi,\varphi)(h)=&\int_{\SL_2(\bq)\backslash \SL_2(\ba)}\overline{\varphi(g)}\sum_{\gamma\in N(F)\backslash \SL_2(F)}\sum_{\xi\in F}\omega(\gamma g, h)\widehat{\phi}(\xi;0,1)dg\\
=&\int_{N(\ba)\backslash \SL_2(\ba)}\overline{W_{\varphi,\psi_{q(\xi)}}(g)}\sum_{\xi\in F}\omega(g,h)\widehat{\phi}(\xi;0,1).
\end{align*}
For $\alpha\in F^\times$ and $f=\Theta_\psi(\phi,\varphi)$, it follows that
\begin{align*}
\ell_{\pi,\psi_\alpha}(f)=&\int_{\ba/F}f(n)\psi(-\alpha n)dn\\
=&\int_{\ba/F}\int_{N(\ba)\backslash \SL_2(\ba)}\overline{W_{\varphi,\psi_{q(\xi)}}(g)}\big[\sum_{\xi\in F}\omega(g,n)\widehat{\phi}(\xi;0,1)\big]\psi(-\alpha n)dgdn\\
=&\int_{N(\ba)\backslash \SL_2(\ba)}\overline{W_{\varphi,\psi_{q(\xi)}}(g)}\big[\sum_{\xi\in F}\omega(g)\widehat{\phi}(\xi;0,1)\int_{\ba/F}\psi\big(-q(ne_0,\xi e_0)-\alpha n\big)dn\big]dg\\
=&\int_{N(\ba)\backslash \SL_2(\ba)}\overline{W_{\varphi,\psi_{\frac{\alpha^2}{4}}}(g)}\omega(g)\widehat{\phi}(-\frac{\alpha}{2};0,1)dg.
\end{align*}.
\end{proof}

Second, we transfer the local Whittaker period functional on $\pi_v$ to an integral over $\wtilde{\SL}_2(F_v)$. It involves subtler convergence issue and more care is needed. We first recall the inner product formula concerning the lifting $\sigma\rar \pi$ and the according normalized local theta correspondences.
\begin{proposition}\cite[Prop. 2.8 (ii)]{qiu14}\label{prop_ip}
Suppose $\sigma\subset \ma_{00}(\wtilde{\SL}_2)$ and $\pi=\Theta(\sigma,\psi)$ is non-zero. For decomposable vectors $\varphi_i\in \sigma$ and $\phi_i\in \mathrm{S}(V_\ba)$ ($i=1,2$), there is
\[
\big(\Theta(\phi_1,\varphi_1),\Theta(\phi_2, \varphi_2)\big)_\pi=\frac{L(\frac{1}{2},\pi)}{\zeta_F(2)}\prod_v \frac{\zeta_{F_v}(2)}{L(\frac{1}{2},\pi_v)}\frac{}{}\int_{\SL_2(F_v)}\overline{(\sigma_v(g_v)\varphi_{1,v},\varphi_{2,v})}(\omega_v(g_v)\phi_{1,v},\phi_{2,v})dg_v.
\]
\end{proposition}
\begin{remark}
The local $\SL_2(F_v)$-integral on the righ hand side of the formula is convergent and equal to $\frac{L(\frac{1}{2},\pi_v)}{\zeta_{F_v}(2)}$ at almost all places, whence almost all factors in the infinite product is $1$.
\end{remark}

By local Howe duality, at each place $v$ of $F$, the $\wtilde{\SL}_2(F_v)\times \PGL_2(F_v)$-invariant homomorphisms $\theta_v:\mathrm{S}(V_{F_v})\rar \sigma_v\otimes \pi_v$ form a vector space of dimension $1$.  We choose $\theta_v$ so that with respect to the associated local theta lifting $\theta_v(\phi_v,\varphi_v)=(\theta_v(\phi_v),\varphi_v)_{\sigma_v}$, there is
\begin{equation*}
\big(\theta_v(\phi_{1,v},\varphi_{1,v}), \theta_v(\phi_{2,v},\varphi_{2,v})\big)_{\pi_v}=\frac{\zeta_{F_v}(2)}{L(\frac{1}{2},\pi_v)}\int_{\SL_2(F_v)}(\omega_v(g_v)\phi_{1,v},\phi_{2,v})\overline{(\sigma_v(g_v)\varphi_{1,v},\varphi_{2,v})_{\sigma_v}}dg_v.
\end{equation*}
Such a normalization by inner product is unique up to a constant of norm $1$. Recall that the global representations $\sigma$ and $\pi$ are restricted tensor products $\sigma=\otimes_v \sigma_v$ and $\pi=\otimes_v \pi_v$ with respect to a choice of spherical vectors $\varphi_{v,0}\in \sigma_v$ and $f_{v,0}\in \pi_v$ for almost all $v$. We require $\theta_v(1_{V(\mo_{F_v})},\varphi_{v,0})=f_{v,0}$ at almost all places, then there is a global constant $c_{\sigma,\pi}$ satisfying
\begin{equation}\label{decomposetheta}
\Theta(\phi,\varphi)=c_{\sigma,\pi}\prod_v \theta_v(\phi_v,\varphi_v).
\end{equation}
The global inner product formula means that $|c_{\sigma,\pi}|^2=\frac{L(\frac{1}{2},\pi)}{\zeta_F(2)}$.

Furthermore, at each place $v$, there is a homeomorphism
\[
N(F_v)\backslash \SL_2(F_v)\rar F_v^2\backslash\{(0,0)\},\quad N(F_v)g_v\rar (0,1)g_v.
\]
So the additive Haar measure on $F_v^2$ at $(0,1)g_v$ can be written as $c_v^\prime d\dot{g}_v$, where $c_v^\prime$ is a local constant depending on the quotient measure $d\dot{g}_v$ on $N(F_v)\backslash \SL_2(F_v)$. Because $\ba^2$, $N(\ba)$, and $\SL_2(\ba)$ are all given the Tamagawa measures and the maximal compact subgroup $K_v$ of $\SL_2(F_v)$ has measure $1$ for almost all $v$, thus $c_v^\prime=\frac{1}{\zeta_{F_v}(2)}$ for almost all $v$ and $\prod_v c_v^\prime=1$.

Lastly,  we recall the normalized local Whittaker period functional on $\pi_v$ in (\ref{nfunctional12}) and also formula (\ref{formula21}), which expresses the inner product on $\sigma_v$ in terms of Whittaker functions.
\begin{lemma}\label{transfer_local}
Write $f_{i,v}=\theta_v(\phi_{i,v},\varphi_{i,v})\in \pi_v$, $i=1,2$, then
\[
\ml_{\pi_v,\psi_{-2,v}}^\sharp(f_{1,v},f_{2,v})=\frac{c_v^\prime c_{\sigma_v,\psi_v}L(1,\pi_v,\ad)}{|2|_vL(\frac{1}{2},\pi_v)}J_v(\varphi_{1,v},\phi_{1,v})\overline{J_v(\varphi_{2,v},\phi_{2,v})},
\]
where $J_v(\varphi_v,\phi_v):=\int_{N(F_v)\backslash \SL_2(F_v)}\overline{W_{\varphi_v,\psi_v}(g_v)}\omega_v(g_v)\widehat{\phi}_v(1;0,1)dg_v$ equals $1$ at almost all places.
\end{lemma}
\begin{proof}
\emph{Step 1.} With the identification $N(F_v)\backslash \SL_2(F_v)\rar F_v^2\backslash\{(0,0)\}$, we first observe that
\begin{align*}
(\phi_{1,v},\phi_{2,v})=&\iint_{F_v\times F_v^2}\widehat{\phi}_{1,v}(a_v;y_v)\overline{\widehat{\phi}_{2,v}(a_v;y_v)}da_vdy_v\\
=&c_v^\prime \iint_{F_v\times [N(F_v)\backslash \SL_2(F_v)]}\omega(g_v^\prime)\widehat{\phi}_{1,v}(a_v;0,1)\overline{\omega_v(g_v^\prime)\widehat{\phi}_{2,v}(a_v ;0,1)}da_vd\dot{g}_v^\prime
\end{align*}
Since $\omega_v(n_v)\widehat{\phi}_v(a_v;0,1)=\widehat{\phi}_v(a_v;0,1)\psi(-2n_va_v)$, the following holds for $\lambda\in \ms(F_v)$:
\begin{align}
\nonumber &\int_{F_v}(\omega_v(n_v)\phi_{1,v},\phi_{2,v})\widehat{\lambda}(-n_v)dn_v\\
=&c_v^\prime \iint_{F_v\times [N(F_v)\backslash \SL_2(F_v)]}\omega(g_v^\prime)\widehat{\phi}_{1,v}(a_v;0,1)\overline{\omega_v(g_v^\prime)\widehat{\phi}_{2,v}(a_v;0,1)}\lambda(-2a_v)da_vd\dot{g}_v^\prime. \label{proofequality1}
\end{align}

\emph{Step 2.} We compute $W_{f_{1,v},f_{2,v},\psi_v}(a_v)$ to get $\ml_{\pi_v,\psi_{-2,v}}^\sharp(f_{1,v},f_{2,v})$. For $\lambda\in C^\infty_c(F_v^\times)$, there is
\begin{align}
\nonumber &\int_{F_v^\times} W_{f_{1,v},f_{2,v},\psi_v}(a_v)\lambda(a_v)da_v\\
\nonumber =&\int_{F_v} (\pi_v(n_v)f_{1,v},f_{2,v})\widehat{\lambda}(-n_v)dn_v\\
\nonumber =&\frac{\zeta_{F_v}(2)}{L(\frac{1}{2},\pi_v)}\int_{F_v}\int_{SL_{2}(F_v)}(\omega_v(g_v,n_v)\phi_{1,v},\phi_{2,v})\overline{(\sigma(g_v)\varphi_{1,v},\varphi_{2,v})}\widehat{\lambda}(-n_v)dg_vdn_v\\
\nonumber =&\frac{\zeta_{F_v}(2)}{L(\frac{1}{2},\pi_v)}\int_{\SL_2(F_v)}\overline{(\sigma_v(g_v)\varphi_{1,v},\varphi_{2,v})}\big(\int_{F_v}\omega_v(g_v,n_v)\phi_{1,v},\phi_{2,v})\widehat{\lambda}(-n_v)dn_v\big)dg_v\\
=&\frac{c_v^\prime\zeta_{F_v}(2)}{L(\frac{1}{2},\pi_v)}\int_{\SL_2(F_v)}\int_{F_v}\int_{N(F_v)\backslash \SL_2(F_v)}\overline{(\sigma_v(g_v)\varphi_{1,v},\varphi_{2,v})}\omega_v(g_v^\prime g_v)\widehat{\phi}_{1,v}(a_v;0,1) \label{wexpression1}\\
\nonumber &\hspace{5cm}\cdot\overline{\omega_v(g_v^\prime)\widehat{\phi}_{2,v}(a_v;0,1)}\lambda(-2a_v)d\dot{g}_v^\prime da_v dg_v. 
\end{align}
In the last equality, equation (\ref{proofequality1}) from Step 1 is used.

Now write $\dot{g}_v^\prime=\smalltwomatrix{b_v}{}{}{b_v^{-1}} k_{2,v}$ with $b_v\in F_v^\times$ and $k_{2,v}\in K_v$. Put $\leftup{k_{2,v}}{\widehat{\phi}}_{2,v}$ for $\omega_v(k_{2,v})\widehat{\phi}_{2,v}$. By making the change of variable $g_v=k_{2,v}^{-1}g_v$, one can rewrite equation (\ref{wexpression1}) as
\begin{align*}
&\int_{F_v^\times} W_{f_{1,v},f_{2,v},\psi_v}(a_v)\lambda(a_v)da_v=\frac{c_v^\prime\zeta_{F_v}(2)}{L(\frac{1}{2},\pi_v)}\int_{K_v}\int_{\SL_2(F_v)}\int_{F_v}\int_{F_v^\times}\overline{(\sigma_v(g_v)\varphi_{1,v},\leftup{k_{2,v}}{\varphi}_{2,v})}\\
&\hspace{2.5cm}\cdot \omega_v(g_v)\widehat{\phi}_{1,v}(a_vb_v ;0,b_v^{-1}) \overline{\leftup{k_{2,v}}{\widehat{\phi}}_{2,v}(a_vb_v ;0,b_v^{-1})}\lambda(-2a_v) |b_v|^{-1}d^\ast b_v da_v dg_vdk_{2,v}.
\end{align*}
The key observation is that the innermost integration over $b_v\in F_v^\times$ yields an $L^1$-function of $(g_v, a_v)$ over $\SL_2(F_v)\times F_v$, as indicated below. Thus, one can change the order of integration of $g_v$ and $a_v$.

\begin{claim*}
The function below is integrable over $(g_v,a_v)\in \SL_2(F_v)\times F_v$,
\[
\beta(g_v,a_v)=\overline{(\sigma_v(g_v)\varphi_{1,v},\varphi_{2,v})}\lambda(-2a_v)\int_{
F_v^\times}\omega_v(g_v)\widehat{\phi}_{1,v}(a_vb_v;0,b_v^{-1}) \overline{\widehat{\phi}_{2,v}(a_vb_v;0,b_v^{-1})} |b_v|^{-1}d^\ast b_v.
\]
\end{claim*}
To verify the claim, write $g_v=\smalltwomatrix{1}{x_v}{}{1}\smalltwomatrix{t_v}{}{}{t_v^{-1}}k_v$, then the $F_v^\times$-integral above is
\begin{equation}\label{l1function}
\int_{
F_v^\times}\omega_v(k_v)\widehat{\phi}_{1,v}(t_va_vb_v;0,b_v^{-1}t_v^{-1}) \overline{\widehat{\phi}_{2,v}(a_vb_v;0,b_v^{-1})} \psi_v(x_va_v^2b_v^2)\chi_{\psi_v}(t_v)|t_v|^{-3/2}|b_v|^{-1}d^\ast b_v.
\end{equation}
We distinguish two cases:
\begin{itemize}
\item[(i)] $F_v$ is nonarchimedean. In this situation, $\beta(g_v, a_v)$ can be shown to be compactly supported. Actually, when $\beta(g_v, a_v)$ is nonzero, the variable $a_v$ can only vary in a compact subset of $F_v^\times$ because $\lambda(\cdot)$ is compactly supported in $F_v^\times$. Furthermore, with the occurrence of $\widehat{\phi}_{2,v}(a_vb_v;0,b_v^{-1})$ in (\ref{l1function}), the variable $b_v$ is forced to vary in a compact subset of $F_v^\times$ too; accordingly, with the occurrence of $\omega_v(k_v)\widehat{\phi}_{1,v}(t_va_vb_v;0,b_v^{-1}t_v^{-1})$ in (\ref{l1function}) and the fact that $\phi_{1,v}$ is $K_v$-finite, one sees that the variable $t_v$ can only vary in a compact subset of $F_v^\times$.

With these considerations, we now check the possible range of the variable $x_v$. Observe that there exists $m>0$ such that as a function of $b_v$,
\[
\omega_v(k_v)\widehat{\phi}_{1,v}(t_va_vb_v;0,b_v^{-1}t_v^{-1}) \overline{\widehat{\phi}_{2,v}(a_vb_v;0,b_v^{-1})}
\]
is constant on each coset of $1+\varpi^m\mo_{F_v}^\times$; this means that for the integrand  in (\ref{l1function}), over each coset of $1+\varpi^m\mo_{F_v}^\times$, only the factor $\psi_v(x_va_v^2b_v^2)$ varies. So there eixsts $m^\prime$ such that when $x_v\not\in \varpi^{m^\prime}\mo_k$, the integrand in (\ref{l1function}) has a zero integral on each coset of $1+\varpi^m\mo_{F_v}^\times$. Thus, the expression in (\ref{l1function}) and the function $\beta(g_v, a_v)$ vanish when $x_v$ is not in the compact subset $\varpi^{m^\prime}\mo_{F_v}$. Therefore, when $\beta(g_v, a_v)$ is nonzero, each of $a_v, k_v, t_v, x_v$ can only vary in a compact subset.

\item[(ii)] $F_v$ is archimedean. First, for the smooth function $\beta(g_v, a_v)$ to be nonzero, it is obvious that $a_v$ can only vary in a compact subset of $F_v^\times$. One can similarly argue that the expression in (\ref{l1function}) is rapidly decreasing when $(x_v, t_v)\rar (\infty,\infty)$ or $(\infty, 0)$. Thus, $\beta(g_v, a_v)$ is integrable.\\
\end{itemize}

We continue the argument before the claim. Changing the order of $g_v$ and $a_v$ leads to
\begin{align*}
&\int_{F_v^\times} W_{f_{1,v},f_{2,v},\psi_v}(a_v)\lambda(a_v)da_v=\frac{c_v^\prime\zeta_{F_v}(2)}{L(\frac{1}{2},\pi_v)}\int_{F_v}\int_{K_v}\int_{\SL_2(F_v)}\int_{F_v^\times}\overline{(\sigma_v(g_v)\varphi_{1,v},\leftup{k_{2,v}}{\varphi}_{2,v})}\\
&\hspace{2cm}\cdot \omega_v(g_v)\widehat{\phi}_{1,v}(a_vb_v;0,b_v^{-1}) \overline{\leftup{k_{2,v}}{\widehat{\phi}}_{2,v}(a_vb_v;0,b_v^{-1})}\lambda(-2a_v) |b_v|^{-1}d^\ast b_v dg_v k_{2,v}da_v.
\end{align*}
Since $\lambda(a_v)$ is arbitrary in $C^\infty_c(F_v^\times)$, it follows that for $a_v\in F_v^\times$,
\begin{align}\label{wfunction}
W_{f_{1,v},f_{2,v},\psi_v}(a_v)=&
\frac{c_v^\prime\zeta_{F_v}(2)}{|2|_vL(\frac{1}{2},\pi_v)}\int_{K_v}\int_{\SL_2(F_v)}\int_{F_v^\times}\overline{(\sigma_v(g_v)\varphi_{1,v},\leftup{k_{2,v}}{\varphi}_{2,v})}\\
\nonumber &\cdot \omega_v(g_v)\widehat{\phi}_{1,v}(-\frac{a_vb_v}{2};0,b_v^{-1}) \overline{\leftup{k_{2,v}}{\widehat{\phi}}_{2,v}(-\frac{a_vb_v}{2};0,b_v^{-1})}|b_v|^{-1}d^\ast b_v dg_v k_{2,v}.
\end{align}

\noindent \emph{Step 3.} Now set $a_v=-2$ in (\ref{wfunction}) and decompose $g_v=n_v\underline{a_v}k_{1,v}$ (with a new $a_v$). There is
\[
W_{f_{1,v},f_{2,v},\psi_v}(-2)=\frac{c_v^\prime\zeta_{F_v}(2)}{|2|_vL(\frac{1}{2},\pi_v)}\int_{K_v\times K_v}\int_{F_v^\times}I_v(\underline{a_v}\circ \leftup{k_{1,v}}{\varphi_{1,v}},\leftup{k_{2,v}}{\varphi_{2,v}}, \underline{a_v}\circ \leftup{k_{1,v}}{\phi_{1,v}},\leftup{k_{2,v}}{\phi_{2,v}})|a_v|^{-2}d^\ast a_v,
\]
with
\[
I_v(\varphi_{1,v},\varphi_{2,v},\phi_{1,v},\phi_{2,v}):=\int_{F_v}\int_{F_v^\times}\overline{(n_v\circ \varphi_{1,v},\varphi_{2,v})}\cdot\omega_v(n_v)\widehat{\phi}_{1,v}(b_v;0,b_v^{-1})\overline{\widehat{\phi}_{2,v}(b_v;0,b_v^{-1})}|b_v|^{-1}d^\ast b_v dn_v.
\]

Note that
\[
\int_{F_v^\times}\omega_v(n_v)\widehat{\phi}_{1,v}(b_v;0,b_v^{-1})\overline{\widehat{\phi}_{2,v}(b_v;0,b_v^{-1})}|b_v|^{-1}d^\ast b_v=\mf_2 \big[\widehat{\phi}_{1,v}(b_v;0,b_v^{-1})\overline{\widehat{\phi}_{2,v}(b_v;0,b_v^{-1})}|b_v|^{-2}\cdot \frac{d^\ast b_v}{d^\times b_v}\big].
\]
So by Lemma ~\ref{identity-mbintegral}, the expression $I_v(-)$ can be written as
\[
I_v(\varphi_{1,v},\varphi_{2,v},\phi_{1,v},\phi_{2,v})=c_{\sigma_v,\psi_v}\int_{F_v^\times} \overline{W_{\varphi_{1,v},\psi_v}}W_{\varphi_{2,v},\psi_v}(\underline{b_v})\widehat{\phi}_{1,v}(b_v;0,b_v^{-1})\overline{\widehat{\phi}_{2,v}(b_v;0,b_v^{-1})}|b_v|^{-3}d^\ast b_v.
\]
It follows that (by applying the change of variable $a_v\rar a_vb_v^{-1}$ in the second line below, there is)
\begin{align*}
&c_{\sigma_v,\psi_v}^{-1}\cdot \int_{F_v^\times}I_v(a_v\circ \varphi_{1,v},\varphi_{2,v},a_v\circ \phi_{1,v},\phi_{2,v})|a_v|^{-2}d^\ast a_v\\
=&\int_{F_v^\times} \overline{W_{\varphi_{1,v},\psi_v}(\underline{a_vb_v})}W_{\varphi_{2,v},\psi_v}(\underline{b_v})\,\, \widehat{\phi}_{1,v}(a_vb_v;0,(a_vb_v)^{-1})\overline{\widehat{\phi}_{2,v}(b_v;0,b_v^{-1})}\\
&\hspace{5cm} \cdot \chi_{\psi_v}(a_v)<a_v,b_v>|a_v|^{-\frac{3}{2}}|b_v|^{-3}d^\ast a_vd^\ast b_v\\
=&\int_{F_v^\times}\overline{W_{\varphi_{1,v},\psi_v}(\underline{a_v})}\omega_v(\underline{a_v})\widehat{\phi}_{1,v}(1;0,1)|a_v|^{-2}d^\ast a_v\cdot \int_{F_v^\times}\overline{W_{\varphi_{1,v},\psi_v}(\underline{b_v})}\omega_v(\underline{b_v})\widehat{\phi}_{1,v}(1;0,1)|b_v|^{-2}d^\ast b_v
\end{align*}

Therefore $W_{f_{1,v},f_{2,v},\psi_v}(-2)=\frac{c_v^\prime c_{\sigma_v,\psi_v}\zeta_{F_v}(2)}{|2|_vL(\frac{1}{2},\pi_v)}J_v(\varphi_{1,v},\phi_{1,v})\overline{J_v(\varphi_{2,v},\phi_{2,v})}$ and we accordingly have
\begin{align*}
\ml_{\pi_v,\psi_{-2,v}}^\sharp(f_{1,v},f_{2,v})=&\frac{L(1,\pi_v,\ad)}{\zeta_{F_v}(2)}W_{f_{1,v},f_{2,v},\psi_v}(-2)=\frac{c_v^\prime c_{\sigma_v,\psi_v}L(1,\pi_v,\ad)}{|2|_vL(\frac{1}{2},\pi_v)}J_v(\varphi_{1,v},\phi_{1,v})\overline{J_v(\varphi_{2,v},\phi_{2,v})}.
\end{align*}
\end{proof}

\begin{proposition}\label{csigma}
$c_{\sigma}=c_\pi$, hence $c_{\sigma}=\frac{1}{2}$.
\end{proposition}
\begin{proof}
Consider decomposable vectors $\varphi_i=\otimes \varphi_{i,v}\in \sigma$ and $\phi_i=\otimes \phi_{i,v}\in \mathrm{S}(V_\ba)$, $i=1,2$. Put $f_i=\Theta(\varphi_i,\phi_i)=\otimes_v f_{i,v}$. Also set $f_{i,v}^\prime=\theta_v(\varphi_{i,v},\phi_{i,v})$, then by (\ref{decomposetheta}), $f_i=c_{\sigma,\pi}(\otimes f_{i,v}^\prime)$.

On one hand, by lemma ~\ref{transfer_global}, there is
\[
\ell_{\pi,\psi_{-2}}(f_1)\overline{\ell_{\pi,\psi_{-2}}(f_2)}=\prod_v J_v(\varphi_{1,v},\phi_{1,v})\overline{J_v(\varphi_{2,v},\phi_{2,v})}.
\]
On the other hand, we can apply lemma ~\ref{transfer_local}, the fact $\prod_v c_v^\prime=1$, and Proposition ~\ref{sl2wpfa} to get
\begin{align*}
\prod_v \ml_{\pi_v,\psi_{-2,v}}^\sharp(f_{1,v},f_{2,v})
=&|c_{\sigma,\pi}|^2\prod_v \ml_{\pi_v,\psi_{-2,v}}^\sharp(f_{1,v}^\prime,f_{2,v}^\prime)\\
=&\frac{L(\frac{1}{2},\pi)}{\zeta_F(2)}\prod_v \frac{c_v^\prime c_{\sigma_v,\psi_v}L(1,\pi_v,\ad)}{|2|_vL(\frac{1}{2},\pi_v)}J_v(\varphi_{1,v},\phi_{1,v})\overline{J_v(\varphi_{2,v},\phi_{2,v})}\\
=&L(\frac{1}{2},\pi)\prod_v \frac{c_{\sigma_v,\psi_v}L(1,\pi_v,\ad)}{\zeta_{F_v}(2)L(\frac{1}{2},\pi_v)}J_v(\varphi_{1,v},\phi_{1,v})\overline{J_v(\varphi_{2,v},\phi_{2,v})}\\
=&\frac{1}{c_\sigma}\frac{L(1,\pi,\ad)}{\zeta_F(2)}\prod_v J_v(\varphi_{1,v},\phi_{1,v})\overline{J_v(\varphi_{2,v},\phi_{2,v})}.
\end{align*}
So $\ell_{\pi,\psi_{-2}}(f_1)\overline{\ell_{\pi,\psi_{-2}}(f_2)}=c_{\sigma}\cdot \frac{\zeta_F(2)}{L(1,\pi,\ad)} \prod_v \ml_{\pi_v,\psi_{-2,v}}^\sharp(f_{1,v},f_{2,v})$. Therefore $c_\sigma=c_\pi$.
\end{proof}

\begin{theorem}
Suppose $\sigma\subset \ma_{00}(\wtilde{\SL}_2)$ and $\ell_\psi$ is non-vanishing on $\sigma$. Write $\pi=\Theta_\psi(\sigma,\psi)$, then
\[
\ell_\psi\otimes \overline{\ell_\psi}=2^{-1}\cdot \frac{L(\frac{1}{2},\pi)\zeta_F(2)}{L(1,\pi,\ad)}\prod_v\ml_{\psi_v}^\sharp.
\]
\end{theorem}

\def\cprime{$'$} \def\cprime{$'$}

\end{document}